\documentclass{article}
\usepackage{amsmath,amssymb,amsthm}
\usepackage{url,color}
\usepackage{mathrsfs}
\usepackage[colorlinks=true]{hyperref}
\usepackage{pdfsync}

\def\R{\textrm{I\kern-0.21emR}}
\def\N{\textrm{I\kern-0.21emN}}

\renewcommand{\geq}{\geqslant}
\renewcommand{\leq}{\leqslant}

\newtheorem{theorem}{Theorem}  
\newtheorem{proposition}{Proposition}

\newtheorem{lemma}{Lemma}

\theoremstyle{definition}\newtheorem{remark}{Remark}

\title{Stabilization of semilinear PDE's, and uniform decay under discretization}

\author{
Emmanuel Tr\'elat\footnote{Sorbonne Universit\'es, UPMC Univ Paris 06, CNRS UMR 7598, Laboratoire Jacques-Louis Lions, Institut Universitaire de France, F-75005, Paris, France (\texttt{emmanuel.trelat@upmc.fr}).}
}

\date{}

\begin{document}

\maketitle

\begin{abstract}
These notes are issued from a short course given by the author in a summer school in Chamb\'ery in June 2015.

We consider general semilinear PDE's and we address the following two questions:\\
1) How to design an efficient feedback control locally stabilizing the equation asymptotically to $0$?\\
2) How to construct such a stabilizing feedback from approximation schemes?

To address these issues, we distinguish between parabolic and hyperbolic semilinear PDE's. 
By parabolic, we mean that the linear operator underlying the system generates an analytic semi-group. By hyperbolic, we mean that this operator is skew-adjoint.

We first recall general results allowing one to consider the nonlinear term as a perturbation that can be absorbed when one is able to construct a Lyapunov function for the linear part. We recall in particular some known results borrowed from the Riccati theory.

However, since the numerical implementation of Riccati operators is computationally demanding, we focus then on the question of being able to design ``simple" feedbacks. For parabolic equations, we describe a method consisting of designing a stabilizing feedback, based on a small finite-dimensional (spectral) approximation of the whole system.
For hyperbolic equations, we focus on simple linear or nonlinear feedbacks and we investigate the question of obtaining sharp decay results.

When considering discretization schemes, the decay obtained in the continuous model cannot in general be preserved for the discrete model, and we address the question of adding appropriate viscosity terms in the numerical scheme, in order to recover a uniform decay. We consider space, time and then full discretizations and we report in particular on the most recent results obtained in the literature.

Finally, we describe several open problems and issues.
\end{abstract}

\newpage
\tableofcontents
\newpage

\section{Introduction and general results}
\subsection{General setting}
Let $X$ and $U$ be Hilbert spaces.
We consider the semilinear control system
\begin{equation}\label{contsys}
\dot y(t) = Ay(t) + F(y(t)) + Bu(t) ,
\end{equation}
where $A:D(A)\rightarrow X$ is an operator on $X$, generating a $C_0$ semi-group, $B\in L(U,D(A^*)')$, and $F:X\rightarrow X$ is a (nonlinear) mapping of class $C^1$, assumed to be Lipschitz on the bounded sets of $X$, with $F(0)=0$. We refer to \cite{CazenaveHaraux,EngelNagel,Trelat_SB,TucsnakWeiss} for well-posedness of such systems (existence, uniqueness, appropriate functional spaces, etc).

We focus on the following two questions:
\begin{enumerate}
\item How to design an efficient feedback control $u=Ky$, with $K\in L(X,U)$, locally stabilizing \eqref{contsys} asymptotically to $0$?
\item How to construct such a stabilizing feedback from approximation schemes?
\end{enumerate}
Moreover, we want the feedback to be as simple as possible, in order to promote a simple implementation. Given the fact that the decay obtained in the continuous setting may not be preserved under discretization, we are also interested in the way one should design a numerical scheme, in order to get a uniform decay for the solutions of the approximate system, i.e., in order to guarantee uniform properties with respect to the discretization parameters $\triangle x$ and/or $\triangle t$.
This is the objective of these notes, to address those issues.

\medskip

Concerning the first point, note that, in general, stabilization cannot be global because there may exist other steady-states than $0$, i.e., $\bar y\in X$ such that $A\bar y+F(\bar y)=0$.
This is why we speak of local stabilization.
Without loss of generality we focus on the steady-state $0$ (otherwise, just design a feedback of the kind $u=K(y-\bar y)$).

\medskip

We are first going to recall well known results on how to obtain local stabilization results, first in finite dimension, and then in infinite dimension.
As a preliminary remark, we note that, replacing if necessary $A$ with $A+dF(0)$, we can always assume, without loss of generality, that $dF(0)=0$, and thus,
\begin{equation*}
\Vert F(y)\Vert=\mathrm{o}(\Vert y\Vert)\quad \textrm{near}\ y=0.
\end{equation*}
Then, a first possibility to stabilize \eqref{contsys} locally at $0$ is to consider the nonlinear term $F(y)$ as a perturbation, that we are going to absorb with a linear feedback. Let us now recall standard results and methods.

\subsection{In finite dimension}
In this subsection, we assume that $X=\R^n$. In that case, $A$ is a square matrix of size $n$, and $B$ is a matrix of size $n\times m$.
Then, as it is well known, stabilization is doable under the Kalman condition
$$
\mathrm{rank}(B,AB,\ldots,A^{n-1}B)=n,
$$
and there are several ways to do it (see, e.g., \cite{Trelat_contopt} for a reference on what follows).

\paragraph{First possible way: by pole-shifting.}
According to the \emph{pole-shifting theorem}, there exists $K$ (matrix of size $m\times n$) such that $A+BK$ is Hurwitz, that is, all its eigenvalues have negative real part.
Besides, according to the \emph{Lyapunov lemma}, there exists a positive definite symmetric matrix $P$ of size $n$ such that
$$
P(A+BK) + (A+BK)^\top P = -I_n.
$$
It easily follows that the function $V$ defined on $\R^n$ by
$$
V(y) = y^\top P y
$$
is a Lyapunov function for the closed-loop system $\dot y(t)=(A+BK)y(t)$.

Now, for the semilinear system \eqref{contsys} in closed-loop with $u(t)=Ky(t)$, we have
$$
\frac{d}{dt} V(y(t)) = -\Vert y(t)\Vert^2 + y(t)^\top P F(y(t)) \leq -C_1 \Vert y(t)\Vert^2 \leq -C_2 V(y(t)),
$$
under an a priori assumption on $y(t)$, for some positive constants $C_1$ and $C_2$, whence the local asymptotic stability.

\paragraph{Second possible way: by Riccati procedure.}
The Riccati procedure consists of computing the unique negative definite solution of the \emph{algebraic Riccati equation}
$$
A^\top E+EA+EBB^\top E=I_n .
$$
Then, we set $u(t)=B^\top Ey(t)$.

Note that this is the control minimizing $\int_0^{+\infty}( \Vert y(t)\Vert^2+\Vert u(t)\Vert^2) dt$ for the control system $\dot y(t)=Ay(t)+Bu(t)$. This is one of the best well known results of the Linear Quadratic theory in optimal control.

Then the function 
$$
V(y) = -y^* E y
$$
is a Lyapunov function, as before.

Now, for the semilinear system \eqref{contsys} in closed-loop with $u(t)=B^\top Ey(t)$, we have
$$
\frac{d}{dt} V(y(t)) = -y(t)^\top ( I_n + EBB^\top E)y(t) - y(t)^\top E F(y(t)) \leq -C_1 \Vert y(t)\Vert^2 \leq -C_2 V(y(t)) ,
$$
under an a priori assumption on $y(t)$, and we easily infer the local asymptotic stability property.

\subsection{In infinite dimension}
\subsubsection{Several reminders}
When the Hilbert space $X$ is infinite-dimensional, several difficulties occur with respect to the finite-dimensional setting.
To explain them, we consider the uncontrolled linear system
\begin{equation}\label{sysA}
\dot y(t)=Ay(t) ,
\end{equation}
with $A:D(A)\rightarrow X$ generating a $C_0$ semi-group $S(t)$.

The first difficulty is that none of the following three properties are equivalent:
\begin{itemize}
\item $S(t)$ is exponentially stable, i.e., there exist $C>0$ and $\delta>0$ such that
$\Vert S(t)\Vert \leq C e^{-\delta t}$, for every $t\geq 0$.
\item the spectral abscissa is negative, i.e., $\sup\{\mathrm{Re}(\lambda)\ \mid\ \lambda\in\sigma(A)\}  < 0$.
\item All solutions of \eqref{sysA} converge to $0$ in $X$, i.e., $S(t)y^0\underset{t\rightarrow +\infty}\longrightarrow 0$, for every $y^0\in X$.
\end{itemize}

For example, if we consider the linear wave equation with local damping
$$ y_{tt} - \triangle y + \chi_\omega y = 0 ,$$
in some domain $\Omega$ of $\R^n$, with Dirichlet boundary conditions, and with $\omega$ an open subset of $\Omega$, then it is always true that any solution $(y,y_t)$ converges to $(0,0)$ in $H^1_0(\Omega)\times L^2(\Omega)$ (see \cite{Dafermos}). Besides, it is known that we have exponential stability if and only if $\omega$ satisfies the Geometric Control Condition (GCC). This condition says, roughly, that any generalized ray of geometric optics must meet $\omega$ in finite time. Hence, if for instance $\Omega$ is a square, and $\omega$ is a small ball in $\Omega$, then GCC does not hold and hence the exponential stability fails, whereas convergence of solutions to the equilibrium is valid.

In general, we always have
$$
\sup\{\mathrm{Re}(\lambda)\ \mid\ \lambda\in\sigma(A)\}
\leq
\inf \{ \mu\in\R \ \mid \exists C>0,\ \Vert S(t)\Vert\leq Ce^{\mu t}\quad \forall t\geq 0\} ,
$$
in other words the spectral abscissa is always less than or equal to the best exponential decay rate. The inequality may be strict, and the equality is referred to as ``spectral growth condition".

\medskip

Let us go ahead by recalling the following known results:
\begin{itemize}
\item \emph{Datko theorem}: $S(t)$ is exponentially stable if and only if,  for every $y^0\in X$, $S(t)y^0$ converges exponentially to $0$.
\item \emph{Arendt-Batty theorem}:
If there exists $M>0$ such that $\Vert S(t)\Vert \leq M$ for every $t\geq 0$, and if $i\R\subset\rho(A)$ (where $\rho(A)$ is the resolvent set of $A$), then $S(t)y^0\underset{t\rightarrow+\infty}\longrightarrow 0$ for every $y^0\in X$.
\item \emph{Huang-Pr\"uss theorem}:
Assume that there exists $M>0$ such that $\Vert S(t)\Vert \leq M$ for every $t\geq 0$. Then $S(t)$ is exponentially stable if and only if $i\R\subset\rho(A)$ and $\sup_{\beta\in\R}\Vert(i\beta\mathrm{id}-A)^{-1}\Vert<+\infty$.
\end{itemize}
Finally, we recall that:
\begin{itemize}
\item Exactly null controllable implies exponentially stabilizable, meaning that there exists $K\in L(X,U)$ such that $A+BK$ generates an exponentially stable $C_0$ semi-group.
\item Approximately controllable does not imply exponentially stabilizable.
\end{itemize}

For all reminders done here, we refer to \cite{CurtainZwart, EngelNagel, Liu, LiuZheng, Zabczyk}.

\subsubsection{Stabilization}
Let us now consider the linear control system
$$\dot y=Ay+Bu$$
and let us first assume that $B\in L(U,X)$, that is, the control operator $B$ is bounded.
We assume that the pair $(A,B)$ is exponentially stabilizable.

\paragraph{Riccati procedure.}
As before, the Riccati procedure consists of finding the unique negative definite solution $E\in L(X)$ of the \emph{algebraic Riccati equation} 
$$
A^* E+EA+EBB^* E=\mathrm{id} ,
$$
in the sense of $\langle (2EA+EBB^* E-\mathrm{id})y,y\rangle=0$, for every $y\in D(A)$, and then of setting $u(t)=B^* Ey(t)$.
Note that this is the control minimizing $\int_0^{+\infty}( \Vert y(t)\Vert^2+\Vert u(t)\Vert^2)\, dt$ for the control system $\dot y=Ay+Bu$.

Then, as before, the function $V(y) = -\langle y, E y\rangle$ is a Lyapunov function for the system in closed-loop $\dot y=(A+BK)y$.

Now, for the semilinear system \eqref{contsys} in closed-loop with $u(t)=B^*Ey(t)$, we have
$$
\frac{d}{dt} V(y(t)) = -\langle y(t) , ( \mathrm{id} + EBB^* E)y(t)\rangle - \langle y(t), E F(y(t))\rangle \leq -C_1 \Vert y(t)\Vert^2 \leq -C_2 V(y(t))
$$
and we infer local asymptotic stability.

\medskip

For $B\in L(U,D(A^*)')$ unbounded, things are more complicated. Roughly, the theory is complete in the parabolic case (i.e., when $A$ generates an analytic semi-group), but is incomplete in the hyperbolic case (see \cite{LasieckaTriggiani1,LasieckaTriggiani2} for details).

\paragraph{Rapid stabilization.}
An alternative method exists in the case where $A$ generates a \emph{group} $S(t)$, and $B\in L(U,D(A^*)')$ an admissible control operator (see \cite{TucsnakWeiss} for the notion of admissibility).
An example covered by this setting is the wave equation with Dirichlet control.

The strategy developed in \cite{Komornik} consists of setting
$$
u = -B^* C_\lambda^{-1} y  \qquad\textrm{with}\quad 
C_\lambda = \int_0^{T+1/2\lambda} f_\lambda(t) S(-t) B B^* S(-t)^* \, dt ,
$$
with $\lambda>0$ arbitrary, $f_\lambda(t) = e^{-2\lambda t}$ if $t\in[0,T]$ and $f_\lambda(t) = 2\lambda e^{-2\lambda T}(T+1/2\lambda-t)$ if $t\in[T,T+1/2\lambda]$. Besides, the function
$$
V(y) = \langle y, C_\lambda^{-1} y \rangle 
$$
is a Lyapunov function (as noticed in \cite{Coron}).
Actually, this feedback yields exponential stability with rate $-\lambda$ for the closed-loop system $\dot y = (A-BB^*C_\lambda^{-1})y$, whence the wording ``rapid stabilization" since $\lambda>0$ is arbitrary.

Then, thanks to that Lyapunov function, the above rapid stabilization procedure applies as well to the semilinear control system \eqref{contsys}, yielding a local stabilization result (with any exponential rate).

\subsection{Existing results for discretizations}
We recall hereafter existing convergence results for space semi-discretizations of the Riccati procedure. We denote by $E_N$ the approximate Riccati solution, expecting that $E_N\rightarrow E$ as $N\rightarrow +\infty$.

One can find in \cite{BanksIto1997, BanksKunisch1984, Gibson1979, KappelSalamon1990,LiuZheng} a general result showing convergence of the approximations $E_N$ of the Riccati operator $E$, under assumptions of \emph{uniform} exponential stabilizability, and of \emph{uniform} boundedness of $E_N$:
$$
\Vert S_{A_N+B_NB_N^\top E_N}(t)\Vert \leq Ce^{-\delta t},\qquad
\Vert E_N\Vert\leq M ,
$$
for every $t\geq 0$ and every $N\in\N^*$, with uniform positive constants $C$, $\delta$ and $M$.

In \cite{BanksIto1997, LasieckaTriggiani1, LiuZheng}, the convergence of $E_N$ to $E$ is proved in the general parabolic case, for unbounded control operators, that is, when $A:D(A)\rightarrow X$ generates an analytic semi-group, $B\in L(U,D(A^*)')$, and $(A,B)$ is exponentially stabilizable.

The situation is therefore definitive in the parabolic setting. In contrast, if the semi-group $S(t)$ is not analytic, the theory is not complete.
Uniform exponential stability is proved under uniform Huang-Pr\"uss conditions in \cite{LiuZheng}. More precisely, it is proved that, given a sequence $(S_n(\cdot))$ of $C_0$ semi-groups on $X_n$, of generators $A_n$, $(S_n(\cdot))$ is uniformly exponentially stable if and only if $i\R\subset\rho(A_n)$ for every $n\in\N$ and
$$
\sup_{\beta\in\R,\ n\in\N} \Vert (i\beta\mathrm{id}-A_n)^{-1}\Vert < +\infty .
$$
This result is used e.g. in \cite{RTT_COCV2007} to prove uniform stability of second-order equations with (bounded) damping and with viscosity term, under uniform gap condition on the eigenvalues.

This result also allows to obtain convergence of the Riccati operators, for second-order systems
$$\ddot y + A y = Bu,$$
with $A:D(A)\rightarrow X$ positive selfadjoint, with compact inverse, and $B\in L(U,X)$ (bounded control operator).

But the approximation theory for general LQR problems remains incomplete in the general hyperbolic case with unbounded control operators, for instance it is not done for wave equations with Dirichlet boundary control.

\subsection{Conclusion}
Concerning implementation issues, solving Riccati equations (or computing Gramians, in the case of rapid stabilization) in large dimension is computationally demanding. In what follows, we would like to find other ways to proceed.

Our objective is therefore to design \emph{simple} feedbacks with \emph{efficient} approximation procedures.

In the sequel, we are going to investigate two situations:
\begin{itemize}
\item \emph{Parabolic} case ($A$ generates an analytic semi-group): we are going to show how to design feedbacks based on a small number of spectral modes.
\item \emph{Hyperbolic} case, i.e., $A=-A^*$ in \eqref{contsys}: we are going to consider two ``"simple" feedbacks:
\begin{itemize}
\item Linear feedback $u=-B^*y$: in that case, if $F=0$ then $\frac{1}{2}\frac{d}{dt}\Vert y(t)\Vert^2 = - \Vert B^*y(t)\Vert^2$.
We will investigate the question of how to ensure uniform exponential decay for approximations.
\item Nonlinear feedback $u=B^*G(y)$: we will as well investigate the question of how to ensure uniform (sharp) decay for approximations.
\end{itemize}
\end{itemize}

\section{Parabolic PDE's}
In this section, we assume that the operator $A$ in \eqref{contsys} generates an analytic semi-group. Our objective is to design stabilizing feedbacks based on a small number of spectral modes.

To simplify the exposition, we consider a 1D semilinear heat equation, and we will comment further on extensions.

\medskip

Let $L>0$ be fixed and let $f:\R\rightarrow\R$ be a function of class $C^2$ such that $f(0)=0$. Following \cite{CoronTrelat_SICON2004}, We consider the 1D semilinear heat equation
\begin{equation} \label{eqcont0}
y_t = y_{xx}+f(y),  \qquad  y(t,0)=0,\ y(t,L)=u(t),
\end{equation}
where the state is $y(t,\cdot):[0,L]\rightarrow\R$ and the control is $u(t)\in \R$.

We want to design a feedback control locally stabilizing \eqref{eqcont0} asymptotically to $0$.
Note that this cannot be global, because we can have other steady-states (a steady-state is a function $y\in C^2(0,L)$ such that $y''(x)+f(y(x))=0$ on $(0,L)$ and $y(0)=0$).
By the way, here, without loss of generality we consider the steady-state $0$.

Let us first note that, for every $T>0$, \eqref{eqcont0} is well posed in the Banach space 
$$
Y_T = L^2(0,T;H^2(0,L))\cap H^1(0,T;L^2(0,L)),
$$
which is continuously embedded in $L^\infty((0,T)\times(0,L))$.\footnote{Indeed, considering $v\in L^2(0,T;H^2(0,L))$ with $v_t\in H^1(0,T;L^2(0,L))$, writing $v=\sum_{j,k}c_{jk}e^{ijt}e^{ikx}$, we have
$$
\sum_{j,k}\vert c_{jk}\vert\leq \left( \sum_{j,k} \frac{1}{1+j^2+k^4} \right)^{1/2} \left( \sum_{j,k} (1+j^2+k^4)\vert c_{jk}\vert^2 \right)^{1/2} ,
$$
and these series converge, whence the embedding, allowing to give a sense to $f(y)$.

Now, if $y_1$ and $y_2$ are solutions of \eqref{eqcont0} on $[0,T]$, then $y_1=y_2$. Indeed,  $v=y_1-y_2$ is solution of $v_t=v_{xx}+a\, v$, $v(t,0)=v(t,L)=0$, $v(0,x)=0$, with $a(t,x)=g(y_1(t,x),y_2(t,x))$ where $g$ is a function of class $C^1$. We infer that $v=0$.}

First of all, in order to end up with a Dirichlet problem, we set 
$$
z(t,x)=y(t,x)-\frac{x}{L}u(t).
$$
Assuming (for the moment) that $u$ is differentiable, we set $v(t)=u'(t)$, and we consider in the sequel $v$ as a control. We also assume that $u(0)=0$. Then we have
\begin{equation}\label{reducedproblem2}
z_t=z_{xx}+f'(0)z+\frac{x}{L}f'(0)u-\frac{x}{L}v+r(t,x),  \qquad
z(t,0)=z(t,L)=0,  
\end{equation}
with $z(0,x)=y(0,x)$ and
\begin{equation*}
r(t,x)= \left(z(t,x)+\frac{x}{L}u(t)\right)^2\int_0^1 (1-s)f''\left(sz(s,x)+s\frac{x}{L}u(s)\right)ds.
\end{equation*}
Note that, given $B>0$ arbitrary, there exist positive constants $C_1$ and $C_2$ such that, if $\vert u(t)\vert\leq B$ and $\Vert z(t,\cdot)\Vert_{L^\infty(0,L)}\leq B$, then
\begin{equation*}
\Vert r(t,\cdot)\Vert_{L^\infty(0,L)}\leq C_1(u(t)^2 +\Vert z(t,\cdot)\Vert_{L^\infty(0,L)}^2) \leq C_2(u(t)^2 +\Vert z(t,\cdot)\Vert_{H^1_0(0,L)}^2) .
\end{equation*}
In the sequel, $r(t,x)$ will be considered as a remainder.

We define the operator $A=\triangle+f'(0)\mathrm{id}$ on $D(A) = H^2(0,L)\cap H_0^1(0,L)$, so that \eqref{reducedproblem2} is written as
\begin{equation}\label{reducedproblem3}
\dot u=v,\qquad z_t=Az+au+bv+r,  \qquad z(t,0)=z(t,L)=0,  
\end{equation}
with $a(x) = \frac{x}{L}f'(0)$ and $b(x)=-\frac{x}{L}$.

Since $A$ is selfadjoint and has a compact resolvent, there exists a Hilbert basis $(e_j)_{j\geq 1}$ of $L^2(0,L)$, consisting of eigenfunctions $e_j\in H^1_0(0,L)\cap C^2([0,L])$ of $A$, associated with eigenvalues $(\lambda_j)_{j\geq 1}$ such that $-\infty<\cdots<\lambda_n<\cdots<\lambda_1$ and $\lambda_n\rightarrow-\infty$ as $n\rightarrow+\infty$.

Any solution $z(t,\cdot)\in H^2(0,L)\cap H^1_0(0,L)$ of \eqref{reducedproblem2}, as long as it is well defined, can be expanded as a series 
$$
z(t,\cdot)=\sum_{j=1}^{\infty}z_j(t)e_j(\cdot)
$$
(converging in $H_0^1(0,L)$), and then we have, for every $j\geq 1$,
\begin{equation*}
\dot z_j(t) = \lambda_j z_j(t) + a_j u(t) + b_j v(t) + r_j(t),
\end{equation*}
with
$$
a_j= \frac{f'(0)}{L}\int_0^L xe_j(x)\, dx, \quad
b_j= -\frac{1}{L}\int_0^L xe_j(x)\, dx, \quad
r_j(t)=\int_0^L r(t,x)e_j(x)\, dx.
$$

Setting, for every $n\in\N^*$,
\begin{equation*}
X_n(t)=\begin{pmatrix} u(t) \\ z_1(t) \\ \vdots \\ z_n(t)
\end{pmatrix}, \
A_n=\begin{pmatrix}
0         &       0         & \cdots &    0           \\
a_1 & \lambda_1 & \cdots &    0           \\
\vdots    &  \vdots         & \ddots &   \vdots       \\
a_n &  0              & \cdots & \lambda_n
\end{pmatrix} , \
B_n=\begin{pmatrix} 1 \\ b_1 \\ \vdots \\ b_n \end{pmatrix} , \ 
R_n(t)=\begin{pmatrix} 0 \\ r_1(t) \\ \vdots \\ r_n(t) \end{pmatrix},
\end{equation*}
we have, then,
\begin{equation*}
\dot X_n(t) = A_nX_n(t) + B_n v(t) + R_n(t).
\end{equation*}

\begin{lemma}
The pair $(A_n,B_n)$ satisfies the Kalman condition.
\end{lemma}

\begin{proof}
We compute
\begin{equation}\label{deter}
\mathrm{det} \left( B_n, A_nB_n, \ldots, A_n^{n}B_n \right)
= \prod_{j=1}^{n}(a_j+\lambda_jb_j)\
\textrm{VdM}(\lambda_1,\ldots,\lambda_n) ,
\end{equation}
where $\mathrm{VdM}(\lambda_1,\ldots,\lambda_n)$ is a Van der Monde determinant, and thus is never equal to zero since the $\lambda_i$, $i=1\ldots n$, are pairwise distinct.
On the other part, using the fact that each $e_j$ is an eigenfunction of $A$ and belongs to $H^1_0(0,L)$, we compute
\begin{equation*}
a_j+\lambda_jb_j
= \frac{1}{L} \int_0^L x ( f'(0)-\lambda_j)e_j(x) \, dx
= -\frac{1}{L} \int_0^L x e_j''(x) \, dx
= - e_j'(L),
\end{equation*}
and this quantity is never equal to zero since $e_j(L)=0$ and $e_j$ is a nontrivial solution of
a linear second-order scalar differential equation.
Therefore the determinant \eqref{deter} is never equal to zero.
\end{proof}

By the pole-shifting theorem, there exists $K_n=\left( k_0,\ldots,k_n \right)$ such that the matrix $A_n+B_nK_n$ has $-1$ as an eigenvalue of multiplicity $n+1$. Moreover, by the Lyapunov lemma, there exists a symmetric positive definite matrix $P_n$ of size $n+1$ such that
\begin{equation*}
P_n\left(A_n+B_nK_n\right) + \left(A_n+B_nK_n\right)^\top P_n = -I_{n+1} .
\end{equation*}
Therefore, the function defined by $V_n(X) = X^\top P_n X$ for any $X\in\R^{n+1}$ is a Lyapunov function for the closed-loop system $\dot X_n(t) = (A_n+B_nK_n)X_n(t)$.

Let $\gamma>0$ and $n\in\N^*$ to be chosen later. For every $u\in\R$ and every $z\in H^2(0,L)\cap H^1_0(0,L)$, we set
\begin{equation}\label{defLyapV1216}
V(u,z)=\gamma\, X_n^\top P_n X_n - \frac{1}{2}\langle z,Az\rangle_{L^2(0,L)}
= \gamma\, X_n^\top P_n X_n - \frac{1}{2}\sum_{j=1}^\infty \lambda_jz_j^2,
\end{equation}
where $X_n\in\R^{n+1}$ is defined by $X_n=(u,z_1,\ldots,z_n)^\top$ and $z_j=\langle z(\cdot),e_i(\cdot)\rangle_{L^2(0,L)}$ for every $j$.

Using that $\lambda_n\rightarrow-\infty$ as $n\rightarrow+\infty$, it is clear that, choosing $\gamma>0$ and $n\in\N^*$ large enough, we have $V(u,z)>0$ for all $(u,z)\in\R\times (H^2(0,L)\cap H^1_0(0,L))\setminus\{(0,0)\}$. More precisely, there exist positive constants $C_3$, $C_4$, $C_5$ and $C_6$ such that
\begin{multline*}
C_3\left( u^2+\Vert z\Vert_{H^1_0(0,L)}^2\right)\ \leq\ V(u,z)\ \leq\ C_4\left( u^2+\Vert z\Vert_{H^1_0(0,L)}^2\right),\\
V(u,z)\ \leq\ C_5\left( \Vert X_n\Vert_2^2 + \Vert Az\Vert_{L^2(0,L)}^2 \right) ,\qquad
\gamma C_6\Vert X_n\Vert_2^2\ \leq\ V(u,z) ,
\end{multline*}
for all $(u,z)\in\R\times (H^2(0,L)\cap H^1_0(0,L))$. Here, $\Vert\ \Vert_2$ designates the Euclidean norm of $\R^{n+1}$.

Our objective is now to prove that $V$ is a Lyapunov function for the system \eqref{reducedproblem3} in closed-loop with the control $v=K_nX_n$.

In what follows, we thus take $v = K_n X_n$ and $u$ defined by $\dot u=v$ and $u(0)=0$.
We compute
\begin{multline}\label{dVuzdt}
\frac{d}{dt} V(u(t),z(t))=-\gamma\,\Vert X_n(t)\Vert_2^2-\Vert Az(t,\cdot)\Vert_{L^2}^2-\langle
Az(t,\cdot),a(\cdot)\rangle_{L^2}u(t) \\ 
\qquad\qquad\qquad\qquad -\langle Az(t,\cdot),b(\cdot)\rangle_{L^2}K_nX_n(t) -\langle Az(t,\cdot),r(t,\cdot)\rangle_{L^2} \\
+\gamma\left(R_n(t)^\top P_nX_n(t)+X_n(t)^\top P_nR_n(t)\right) .
\end{multline}
Let us estimate the terms at the right-hand side of \eqref{dVuzdt}. Under the a priori estimates $\vert u(t)\vert\leq B$ and $\Vert z(t,\cdot)\Vert_{L^\infty(0,L)}\leq B$, there exist positive constants $C_7$, $C_8$ and $C_9$ such that
\begin{equation*}
\begin{split}
& \vert \langle Az,a\rangle_{L^2}u \vert+\vert \langle Az,b\rangle_{L^2}K_nX_n \vert \leq \frac{1}{4} \Vert Az\Vert_{L^2}^2+C_7\Vert X_n\Vert_2^2 , \\
& \vert\langle Az,r\rangle_{L^2}\vert \leq \frac{1}{4} \Vert Az\Vert_{L^2}^2+C_8V^2 ,\qquad
\Vert R_n\Vert_\infty \leq \frac{C_2}{C_3} V , \\
& \vert\gamma\left(R_n^\top P_nX_n+X_n^\top P_nR_n\right)\vert \leq \frac{C_2}{C_3\sqrt{C_6}}\sqrt{\gamma}\,V^{3/2} .
\end{split}
\end{equation*}
We infer that, if $\gamma>0$ is large enough, then there exist positive constants $C_{10}$ and $C_{11}$ such that $\frac{d}{dt} V \leq -C_{10}V+C_{11}V^{3/2}$.
We easily conclude the local asymptotic stability of the system \eqref{reducedproblem3} in closed-loop with the control $v=K_nX_n$.

\begin{remark}
Of course, the above local asymptotic stability may be achieved with other procedures, for instance, by using the Riccati theory. However, the procedure developed here is much more efficient because it consists of stabilizing a finite-dimensional part of the system, mainly, the part that is not naturally stable. We refer to \cite{CoronTrelat_SICON2004} for examples and for more details. Actually, we have proved in that reference that, thanks to such a strategy, we can pass from any steady-state to any other one, provided that the two steady-states belong to a same connected component of the set of steady-states: this is a partially global exact controllability result.
\end{remark}

The main idea used above is the following fact, already used in the remarkable early paper \cite{Russell}. Considering the linearized system with no control, we have an infinite-dimensional linear system that can be split, through a spectral decomposition, in two parts: the first part is finite-dimensional, and consists of all spectral modes that are unstable (meaning that the corresponding eigenvalues have nonnegative real part); the second part is infinite-dimensional, and consists of all spectral modes that are asymptotically stable (meaning that the corresponding eigenvalues have negative real part).
The idea used here then consists of focusing on the finite-dimensional unstable part of the system, and to design a feedback control in order to stabilize that part. Then, we plug this control in the infinite-dimensional system, and we have to check that this feedback indeed stabilizes the whole system (in the sense that it does not destabilize the other infinite-dimensional part). This is the role of the Lyapunov function $V$ defined by \eqref{defLyapV1216}.

The extension to general systems \eqref{contsys} is quite immediate, at least in the parabolic setting under appropriate spectral assumptions (see \cite{SchmidtTrelat} for Couette flows and \cite{CoronTrelatVazquez} for Navier-Stokes equations).

But it is interesting to note that it does not work only for parabolic equations: this idea has been as well used in \cite{CoronTrelat_CCM2006} for the 1D semilinear equation
$$
y_{tt} = y_{xx}+f(y),  \qquad  y(t,0)=0,\ y_x(t,L)=u(t),
$$
with the same assumptions on $f$ as before. We first note that, if $f(y)=cy$ is linear (with $c\in L^\infty(0,L)$), then, setting $u(t) = -\alpha y_t(t,L)$ with $\alpha>0$ yields an exponentially decrease of the energy $\int_0^L ( y_t(t,x)^2+y_x(t,x)^2 )\, dt$, and moreover, the eigenvalues of the corresponding operator have a real part tending to $-\infty$ as $\alpha$ tends to $1$. Therefore, in the general case, if $\alpha$ is sufficiently close to $1$ then at most a finite number of eigenvalues may have a nonnegative real part. Using a Riesz spectral expansion, the same kind of method as the one developed above can therefore be applied, and yields a feedback based on a finite number of modes, that stabilizes locally the semilinear wave equation, asymptotically to equilibrium.

\section{Hyperbolic PDE's}
In this section, we assume that the operator $A$ in \eqref{contsys} is skew-adjoint, that is,
$$
A^*=-A,\quad D(A^*)=D(A).
$$

Let us start with a simple remark. If $F=0$ (linear case), then, choosing the very simple linear feedback $u=-B^*y$ and setting $V(y) = \frac{1}{2}\frac{d}{dt}\Vert y\Vert_X^2$, we have
$$
\frac{d}{dt} V(y(t)) = - \Vert B^*y(t)\Vert_X^2 \leq 0 ,
$$
and then we expect that, under reasonable assumptions, we have exponential asymptotic stability (and this will be the case under observability assumptions, as we are going to see).

Now, if we choose a nonlinear feedback $u=B^*G(y)$, we ask the same question: what are sufficient conditions ensuring asymptotic stability, and if so, with which sharp decay?

Besides, we will investigate the following important question: how to ensure \emph{uniform} properties when discretizing?

\subsection{The continuous setting}
\subsubsection{Linear case}
In this subsection, we assume that $F=0$ (linear case), and we assume that $B$ is bounded. Taking the linear feedback $u=-B^*y$ as said above, we have the closed-loop system $\dot y = Ay - BB^* y$. For convenience, in what follows we rather write this equation in the form (more standard in the literature)
\begin{equation}\label{syslinear}
\dot y(t) + Ay(t) + By(t) = 0,
\end{equation}
where $A$ is a densely defined skew-adjoint operator on $X$ and $B$ is a bounded nonnegative selfadjoint operator on $X$ (we have just replaced $A$ with $-A$ and $BB^*$ with $B$).

We start hereafter with the question of the exponential stability of solutions of \eqref{syslinear}.

\paragraph{Equivalence between observability and exponential stability.}
The following result is a generalization of the main result of \cite{Haraux}.

\begin{theorem}\label{thm_Haraux}
Let $X$ be a Hilbert space, let $A:D(A)\rightarrow X$ be a densely defined skew-adjoint operator, let $B$ be a bounded selfadjoint nonnegative operator on $X$.
We have equivalence of:
\begin{enumerate}
\item There exist $T>0$ and $C>0$ such that every solution of the conservative equation 
\begin{equation}\label{lin_conservative}
\dot \phi(t)+A\phi(t)=0
\end{equation}
satisfies the observability inequality
$$\Vert \phi(0)\Vert_X^2 \leq C\int_0^{T_0}\Vert B^{1/2}\phi(t)\Vert_X^2 \, dt .$$
\item There exist $C_1>0$ and $\delta>0$ such that every solution of the damped equation
\begin{equation}\label{lin_damped}
\dot y(t) + Ay(t) + By(t) = 0
\end{equation}
satisfies
$$E_y(t) \leq C_1 E_y(0)\mathrm{e}^{-\delta t} ,$$
where 
$E_y(t) = \frac{1}{2} \Vert y(t)\Vert_X^2 $.
\end{enumerate}
\end{theorem}

\begin{proof}
Let us first prove that the first property implies the second one: we want to prove that every solution of \eqref{lin_damped} satisfies
$$
E_y(t) = \frac{1}{2} \Vert y(t)\Vert_X^2 \leq E_y(0)\mathrm{e}^{-\delta t} = \frac{1}{2} \Vert y(0)\Vert_X^2  \mathrm{e}^{-\delta t}.
$$
Consider $\phi$ solution of \eqref{lin_conservative} with $\phi(0)=y(0)$.
Setting $\theta=y-\phi$, we have 
$$
\dot\theta+A\theta+By=0,\quad \theta(0)=0.
$$
Then, taking the scalar product with $\theta$, since $A$ is skew-adjoint, we get 
$\langle\dot\theta+By,\theta\rangle_X=0$.
But, setting $E_\theta(t)=\frac{1}{2} \Vert \theta(t)\Vert_X^2 $, we have
$ \dot E_\theta = -\langle By,\theta\rangle_X$.
Then, integrating a first time over $[0,t]$, and then a second time over $[0,T]$, since $E_\theta(0)=0$, we get
$$
\int_0^T E_\theta(t)\,dt = - \int_0^T\int_0^t \langle By(s),\theta(s)\rangle_X \,ds\, dt \\
= -\int_0^T (T-t) \langle B^{1/2}y(t),B^{1/2}\theta(t)\rangle_X\, dt ,
$$
where we have used the Fubini theorem.
Hence, thanks to the Young inequality $ab\leq\frac{\alpha}{2}a^2+\frac{1}{2\alpha}b^2$ with $\alpha=2$, we infer that
\begin{equation*}
\begin{split}
\frac{1}{2}\int_0^T  \Vert \theta(t)\Vert_X^2 \, dt
&\leq T \Vert B^{1/2}\Vert \int_0^T \Vert  B^{1/2}y(t)\Vert_X\Vert\theta(t)\Vert_X \, dt \\
&\leq T^2 \Vert B^{1/2}\Vert^2 \int_0^T \Vert  B^{1/2}y(t)\Vert_X^2 \, dt +  \frac{1}{4}\int_0^T \Vert\theta(t)\Vert_X^2 \, dt,
\end{split}
\end{equation*}
and therefore,
$$
\int_0^T \Vert\theta(t)\Vert_X^2 \, dt \leq 4 T^2 \Vert B^{1/2}\Vert_X^2 \int_0^T \Vert  B^{1/2}y(t)\Vert_X^2 \, dt.
$$
Now, since $\phi=y-\theta$, it follows that
\begin{equation*}
\begin{split}
\int_0^T \Vert B^{1/2}\phi(t)\Vert_X^2 \, dt 
& \leq 2 \int_0^T \Vert B^{1/2}y(t)\Vert_X^2 \, dt + 2\int_0^T \Vert B^{1/2}\theta(t)\Vert_X^2 \, dt \\
& \leq (2+8 T^2 \Vert B^{1/2}\Vert^4) \int_0^T \Vert B^{1/2} y(t)\Vert_X^2 \, dt.
\end{split}
\end{equation*}
Finally, since
$$
E_y(0)=E_\phi(0)=\frac{1}{2} \Vert \phi(0)\Vert_X^2 \leq \frac{C}{2}\int_0^{T}\Vert B^{1/2}\phi(t)\Vert_X^2 \, dt
$$
it follows that
$$
E_y(0) \leq C (1+4 T^2 \Vert B^{1/2}\Vert^4) \int_0^T \Vert B^{1/2} y(t)\Vert_X^2 \, dt.
$$
Besides, one has $E_y'(t)=-\Vert B^{1/2} y(t)\Vert_X^2$, and then
$\int_0^T \Vert B^{1/2} y(t)\Vert_X^2 dt=E_y(0)-E_y(T)$.
Therefore
$$E_y(0) \leq C (1+4 T^2 \Vert B^{1/2}\Vert^4) (E_y(0)-E_y(T))=C_1(E_y(0)-E_y(T))$$
and hence
$$
E_y(T)\leq \frac{C_1-1}{C_1} E_y(0) = C_2 E_y(0),
$$
with $C_2<1$.

Actually this can be done on every interval $[kT,(k+1)T]$, and it yields
$E_y((k+1)T)\leq C_2 E_y(kT)$ for every $k\in\N$, and hence $E_y(kT)\leq E_y(0)C_2^k$.

For every $t\in [kT,(k+1)T)$, noting that $k=\left[\frac{t}{T}\right]> \frac{t}{T}-1$, and that $\ln\frac{1}{C_2}>0$, it follows that
$$
C_2^k=\exp(k\ln C_2)=\exp(-k\ln\frac{1}{C_2})\leq
\frac{1}{C_2}\exp(\frac{-\ln\frac{1}{C_2}}{T} t)
$$
and hence $E_y(t)\leq E_y(kT)\leq \delta E_y(0) \exp(-\delta t)$ for some $\delta>0$.

\medskip

Let us now prove the converse: assume the exponential decrease of solutions of \eqref{lin_damped}, and let us prove the observability property for solutions of \eqref{lin_conservative}.

From the exponential decrease inequality, one has
\begin{equation}\label{orl1841}
\int_0^T \Vert B^{1/2} y(t)\Vert_X^2 \, dt = E_y(0)-E_y(T) \geq (1-C_1\mathrm{e}^{-\delta T})E_y(0) = C_2 E_y(0),
\end{equation}
and for $T>0$ large enough there holds $C_2=1-C_1\mathrm{e}^{-\delta T}>0$.

Then we make the same proof as before, starting from \eqref{lin_conservative}, that we write in the form 
$$
\dot\phi+A\phi +B\phi=B\phi,
$$
and considering the solution of \eqref{lin_damped} with $y(0)=\phi(0)$.
Setting $\theta=\phi-y$, we have 
$$
\dot\theta+A\theta+B\theta=B\phi, \quad \theta(0)=0 .
$$
Taking the scalar product with $\theta$, since $A$ is skew-adjoint, we get
$\langle\dot\theta+B\theta,\theta\rangle_X=\langle B\phi,\theta\rangle_X$, 
and therefore
$$\dot E_\theta +\langle B\theta,\theta\rangle_X = \langle B\phi,\theta\rangle_X.$$
Since $\langle B\theta,\theta\rangle_X=\Vert B^{1/2}\theta\Vert_X\geq 0$, it follows that
$ \dot E_\theta \leq \langle B\phi,\theta\rangle_X$.
As before we apply $\int_0^T\int_0^t$ and hence, since $E_\theta(0)=0$,
$$
\int_0^T E_\theta(t) \, dt \leq \int_0^T\int_0^t \langle B\phi(s),\theta(s)\rangle_X \, ds\, dt = \int_0^T (T-t) \langle B^{1/2}\phi(t),B^{1/2}\theta(t)\rangle_X \, dt.
$$
Thanks to the Young inequality, we get, exactly as before,
\begin{equation*}
\begin{split}
\frac{1}{2}\int_0^T  \Vert \theta(t)\Vert_X^2 \, dt
&\leq T \Vert B^{1/2}\Vert \int_0^T \Vert  B^{1/2}\phi(t)\Vert_X \Vert\theta(t)\Vert_X \, dt \\
&\leq T^2 \Vert B^{1/2}\Vert_X^2 \int_0^T \Vert  B^{1/2}\phi(t)\Vert_X^2 \,dt +  \frac{1}{4}\int_0^T \Vert\theta(t)\Vert_X^2 \, dt ,
\end{split}
\end{equation*}
and finally,
$$
\int_0^T \Vert\theta(t)\Vert_X^2 \, dt \leq 4 T^2 \Vert B^{1/2}\Vert_X^2 \int_0^T \Vert  B^{1/2}\phi(t)\Vert_X^2 \,dt.
$$
Now, since $y=\phi-\theta$, it follows that
\begin{equation*}
\begin{split}
\int_0^T \Vert B^{1/2} y(t)\Vert_X^2 \, dt 
& \leq 2 \int_0^T \Vert B^{1/2} \phi(t)\Vert_X^2 \, dt + 2\int_0^T \Vert B^{1/2}\theta(t)\Vert_X^2 \, dt \\
& \leq (2+8 T^2 \Vert B^{1/2}\Vert^4) \int_0^T \Vert B^{1/2} \phi(t)\Vert_X^2 \, dt
\end{split}
\end{equation*}
Now, using \eqref{orl1841} and noting that $E_y(0)=E_\phi(0)$,
we infer that
$$ C_2 E_\phi(0) \leq (2+8 T^2 \Vert B^{1/2}\Vert^4) \int_0^T \Vert B^{1/2} \phi(t)\Vert_X^2 dt.$$
This is the desired observability inequality.
\end{proof}

\begin{remark}
This result says that the observability property for the linear conservative equation \eqref{lin_conservative} is equivalent to the exponential stability property for the linear damped equation \eqref{lin_damped}. This result has been written in \cite{Haraux} for second-order equations, but the proof works exactly in the same way for more general first-order systems, as shown here. More precisely, the statement done in \cite{Haraux} for second-order equations looks as follows:
\begin{quote}
{\it
We have equivalence of:
\begin{enumerate}
\item There exist $T>0$ and $C>0$ such that every solution of
$$
\ddot\phi(t)+A\phi(t)=0\qquad\qquad\textrm{(conservative equation)}
$$
satisfies
$$\Vert A^{1/2}\phi(0)\Vert_X^2+\Vert\dot\phi(0)\Vert_X^2 \leq C\int_0^{T_0}\Vert B^{1/2}\dot\phi(t)\Vert_X^2 dt .$$
\item There exist $C_1>0$ and $\delta>0$ such that every solution of
$$\ddot y(t)+Ay(t)+B\dot y(t)=0\qquad\qquad\textrm{(damped equation)}$$
satisfies
$$E_y(t) \leq C_1 E_y(0)\mathrm{e}^{-\delta t} ,$$
where
$$E_y(t) = \frac{1}{2}\left( \Vert A^{1/2}y(t)\Vert_X^2 +\Vert\dot y(t)\Vert_X^2\right).$$
\end{enumerate}
}
\end{quote}
\end{remark}

\begin{remark}
A second remark is that the proof uses in a crucial way the fact that the operator $B$ is bounded. We refer to \cite{AmmariTucsnak} for a generalization for unbounded operators with degree of unboudedness $\leq 1/2$ (i.e., $B\in L(U,D(A^{1/2})')$), and only for second-order equations, with a proof using Laplace transforms, and under a condition on the unboundedness of $B$ that is not easy to check (related to ``hidden regularity" results), namely,
$$
\forall\beta>0\qquad \sup_{\mathrm{Re}(\lambda)=\beta} \Vert B^* \lambda (\lambda^2I+A)^{-1}B \Vert_{L(U)} < +\infty.
$$
For instance this works for waves with a nonlocal operator $B$ corresponding to a Dirichlet condition, in the state space $L^2\times H^{-1}$, but not for the usual Neumann one, in the state space $H^1\times L^2$ (except in 1D).
\end{remark}

\subsubsection{Semilinear case}
In the case with a nonlinear feedback, still in order to be in agreement with standard notations used in the existing literature, we rather write the equation in the form $\dot u + Au + F(u) = 0$, where $u$ now designates the solution (and not the control).

Therefore, from now on and throughout the rest of the paper, we consider the differential system
\begin{equation}\label{main_eq}
u'(t) + A u(t) + B F(u(t))  = 0 ,
\end{equation}
with $A:D(A)\rightarrow X$ a densely defined skew-adjoint operator, $B:X\rightarrow X$ a nontrivial bounded selfadjoint nonnegative operator, and $F:X\rightarrow X$ a (nonlinear) mapping assumed to be Lipschitz continuous on bounded subsets of $X$.
This is the framework and notations adopted in \cite{AlabauPrivatTrelat}.

If $F=0$ then the system \eqref{main_eq} is purely conservative, and $\Vert u(t)\Vert_X=\Vert u(0)\Vert_X$ for every $t\geq 0$. 
If $F\neq 0$ then the system \eqref{main_eq} is expected to be dissipative if the nonlinearity $F$ has ``the good sign". Along any solution of \eqref{main_eq} (while it is well defined), the derivative with respect to time of the energy $E_u(t) = \frac{1}{2}\Vert u(t)\Vert_X^2$ is
\begin{equation*}
E_u'(t) = -\langle u(t),BF(u(t))\rangle_X = -\langle B^{1/2}u(t),B^{1/2}F(u(t))\rangle_X .
\end{equation*}
In the sequel, we will make appropriate assumptions on $B$ and on $F$ ensuring that $E_u'(t)\leq 0$. It is then expected that the solutions are globally well defined and that their energy  decays asymptotically to $0$ as $t\rightarrow +\infty$.

\medskip

We make the following assumptions.
\begin{itemize}
\item For every $u\in X$
\begin{equation*}
\langle u, B F(u)\rangle_X \geq 0.
\end{equation*}
This assumption implies that $E_u'(t)\leq 0$.
\end{itemize}

The spectral theorem applied to the bounded nonnegative selfadjoint operator $B$ implies that $B$ is unitarily equivalent to a multiplication: there exist a probability space $(\Omega,\mu)$, a real-valued bounded nonnegative measurable function $b$ defined on $X$, and an isometry $U$ from $L^2(\Omega,\mu)$ into $X$, such that $U^{-1} B U f = b f$ for every $f\in L^2(\Omega,\mu)$.

Now, we define the (nonlinear) mapping $\rho: L^2(\Omega,\mu)\rightarrow L^2(\Omega,\mu)$ by
$$
\rho(f) = U^{-1} F( Uf ).
$$
We make the following assumptions on $\rho$:
\begin{itemize}
\item $\rho(0)=0$ and $f \rho(f) \geq 0$ for every $f\in L^2(\Omega,\mu)$.
\item There exist $c_1>0$ and $c_2>0$ such that, for every $f\in L^\infty(\Omega,\mu)$,
\begin{equation*}
\begin{split}
c_1\, g(\vert f(x)\vert) &\leq \vert\rho(f)(x)\vert \leq c_2\, g^{-1}(\vert f(x)\vert) \quad \textrm{for almost every}\ x\in\Omega\ \textrm{such that}\ \vert f(x)\vert\leq 1,    \\
c_1\, \vert f(x)\vert & \leq \vert\rho(f)(x)\vert \leq c_2 \, \vert f(x)\vert\qquad\quad\ \textrm{for almost every}\ x\in\Omega\ \textrm{such that}\ \vert f(x)\vert\geq 1 ,
\end{split}
\end{equation*}
where $g$ is an increasing odd function of class $C^1$ such that $g(0)=g'(0)=0$, $sg'(s)^2/g(s)\rightarrow 0$ as $s\rightarrow 0$, and such that the function $H$ defined by $H(s)=\sqrt{s}g(\sqrt{s})$, for every $s\in [0,1]$, is strictly convex on $[0,s_0^2]$ for some $s_0\in(0,1]$.

This assumption is issued from \cite{alabau_AMO2005} where the optimal weight method has been developed.

Examples of such functions $g$ are given by
$$
g(s) = s / \ln^{p}(1/s),\quad s^p,\quad e^{-1/s^2},\quad s^p\ln^q(1/s),\quad e^{-\ln^p(1/s)}.
$$
\end{itemize}

We define the function $\widehat H$ on $\R$ by $\widehat H(s)=H(s)$ for every $s\in[0,s_0^2]$ and by $\widehat H(s)=+\infty$ otherwise. 
We define the function $L$ on $[0,+\infty)$ by $L(0)=0$ and, for $r>0$, by
\begin{equation*}
L(r)= \frac{\widehat H^*(r)}{r} = \frac{1}{r} \sup_{s\in \R}\left(rs-\widehat H(s)\right),
\end{equation*}
where $\widehat H^*$ is the convex conjugate of $\widehat H$.
By construction, the function $L:[0,+\infty)\rightarrow[0,s_0^2)$ is continuous and increasing.

We define $\Lambda_H:(0,s_0^2]\rightarrow(0,+\infty)$ by $\Lambda_H(s)=H(s)/sH'(s)$, and we set
\begin{equation*}
\forall s\geq 1/H'(s_0^2)\qquad 
\psi(s)=\frac{1}{H'(s_0^2)}+\int_{1/s}^{H'(s_0^2)}\frac{1}{v^2(1-\Lambda_H((H')^{-1}(v)))}\,dv.
\end{equation*}
The function $\psi:[1/H'(s_0^2),+\infty)\rightarrow[0,+\infty)$ is continuous and increasing.

\medskip

Hereafter, we use the notations $\lesssim$ and $\simeq$ in the estimates, with the following meaning.
Let $\mathcal{S}$ be a set, and let $F$ and $G$ be nonnegative functions defined on $\R\times\Omega\times\mathcal{S}$.
The notation $F\lesssim G$ (equivalently, $G\gtrsim F$) means that there exists a constant $C>0$, only depending on the function $g$ or on the mapping $\rho$, such that $F(t,x,\lambda)\leq C G(t,x,\lambda)$ for all $(t,x,\lambda)\in\R\times\Omega\times\mathcal{S}$. 
The notation $F_1\simeq F_2$ means that $F_1\lesssim F_2$ and $F_1\gtrsim F_2$.

In the sequel, we choose $\mathcal{S} = X$, or equivalently, using the isometry $U$, we choose $\mathcal{S} = L^2(\Omega,\mu)$, so that the notation $\lesssim$ designates an estimate in which the constant does not depend on $u\in X$, or on $f\in L^2(\Omega,\mu)$, but depends only on the mapping $\rho$. We will use these notations to provide estimates on the solutions $u(\cdot)$ of \eqref{main_eq}, meaning that the constants in the estimates do not depend on the solutions.

\begin{theorem}[\cite{AlabauPrivatTrelat}]\label{thm_continuous}
In addition to the above assumptions, we assume that there exist $T>0$ and $C_T>0$ such that 
\begin{equation*}
C_T \Vert \phi(0)\Vert_X^2\leq \int_0^T \Vert B^{1/2} \phi(t) \Vert_X^2 \,dt, 
\end{equation*}
for every solution of $\phi'(t)+A\phi(t)=0$ (observability inequality for the linear conservative equation).

Then, for every $u_0\in X$, there exists a unique solution 
$u(\cdot)\in C^0(0,+\infty;X)\cap C^1(0,+\infty;D(A)')$
of \eqref{main_eq} such that $u(0)=u_0$.\footnote{Here, the solution is understood in the weak sense (see \cite{CazenaveHaraux,EngelNagel}), and $D(A)'$ is the dual of $D(A)$ with respect to the pivot space $X$. If $u_0\in D(A)$, then 
$u(\cdot)\in C^0(0,+\infty;D(A))\cap C^1(0,+\infty;X)$.}
Moreover, the energy of any solution satisfies
\begin{equation*}
E_u(t) \lesssim T \max(\gamma_1,E_u(0)) \, L\left( \frac{1}{\psi^{-1}( \gamma_2 t )} \right) ,
\end{equation*}
for every time $t\geq 0$, with $\gamma_1 \simeq \Vert B\Vert /\gamma_2$ and $\gamma_2 \simeq C_T/ T( T^2\Vert B^{1/2}\Vert^4+1)$.
If moreover 
\begin{equation}\label{condlimsup}
\limsup_{s\searrow 0}\Lambda_H(s)<1,
\end{equation}
then we have the simplified decay rate 
\begin{equation*}
E_{u}(t)\lesssim T \max(\gamma_1, E_u(0)) \, (H')^{-1}\left(\frac{\gamma_3}{t}\right),
\end{equation*}
for every time $t\geq 0$, for some positive constant $\gamma_3\simeq 1$.
\end{theorem}

Note the important fact that this result gives \emph{sharp decay rates} (see Table \ref{table1} for examples).

\begin{table}[h]
\begin{center}
\begin{tabular}{|c|c|c|}
\hline
$g(s)$ & $\Lambda_H(s)$ & decay of $E(t)$ \\
\hline\hline
$s / \ln^{p}(1/s)$, $p>0$ & $\displaystyle \limsup_{x\searrow 0}\Lambda_H(s)=1$ & $e^{-t^{1/(p+1)}}/t^{1/(p+1)}$ \\
\hline
$s^p$ on $[0,s_0^2]$, $p>1$ & $\Lambda_H(s)\equiv \frac{2}{p+1}<1$ & $t^{-2/(p-1)}$ \\
\hline
$e^{-1/s^2}$ & ${\displaystyle \lim_{s\searrow 0}\Lambda_H(s)}=0$ & $1/\ln(t)$ \\
\hline
$s^p\ln^q(1/s)$, $p>1$, $q>0$ & ${\displaystyle \lim_{s\searrow 0}\Lambda_H(s)}=\frac{2}{p+1}<1$ & $t^{-2/(p-1)}\ln^{-2q/(p-1)}(t)$ \\
\hline
$e^{-\ln^p(1/s)}$, $p>1$ & ${\displaystyle \lim_{s\searrow 0}\Lambda_H(s)}=0$ & $e^{-2\ln^{1/p}(t)}$ \\
\hline
\end{tabular}
\end{center}
\caption{Examples}\label{table1}
\end{table}

Theorem \ref{thm_continuous} improves and generalizes to a wide class of equations the main result of \cite{alabau_ammari} in which the authors dealt with locally damped wave equations.
Examples of applications are given in \cite{AlabauPrivatTrelat}, that we mention here without giving the precise framework, assumptions and comments:
\begin{itemize}
\item Schr\"odinger equation with nonlinear damping (nonlinear absorption):
$$
i\partial_t u(t,x) + \triangle u(t,x) + i b(x) u(t,x)\rho(x,\vert u(t,x)\vert)=0.
$$
\item Wave equation with nonlinear damping:
$$
\partial_{tt}u(t,x) - \triangle u(t,x) + b(x) \rho(x,\partial_t u(t,x))=0.
$$
\item Plate equation with nonlinear damping:
$$
\partial_{tt}u(t,x) + \triangle^2 u(t,x) + b(x) \rho(x,\partial_t u(t,x))=0.
$$
\item Transport equations with nonlinear damping:
$$
\partial_t u(t,x) + \mathrm{div}(v(x) u(t,x)) + b(x)\rho(x,u(t,x)) = 0,\qquad x\in\mathbb{T}^n,
$$
with $\mathrm{div}(v)=0$.
\item Dissipative equations with nonlocal terms:
$$
\partial_t f + v.\nabla_x f = \rho(f) ,
$$
with kernels $\rho$ satisfying the sign assumption $f\rho(f)\geq 0$.
\end{itemize}

\begin{proof}
It is interesting to quickly give the main steps of the proof.
\begin{itemize}
\item \textbf{Step 1}. Comparison of the nonlinear equation with the linear damped model:\\
Prove that the solutions of
$$u'(t)+Au(t)+BF(u(t))=0,$$
$$z'(t)+Az(t)+Bz(t)=0,\qquad z(0)=u(0),$$
satisfy
$$
\int_0^T \Vert B^{1/2} z(t)\Vert_X^2\, dt \leq 2 \int_0^T \left( \Vert B^{1/2} u(t)\Vert_X^2 + \Vert B^{1/2} F(u(t))\Vert_X^2 \right) dt.
$$
\item \textbf{Step 2}. Comparison of the linear damped equation with the conservative linear equation:\\
Prove that the solutions of
$$z'(t)+Az(t)+Bz(t)=0,\qquad z(0)=u(0),$$
$$\phi'(t)+A\phi(t)=0,\qquad  \phi(0)= u(0),$$
satisfy
$$
\int_0^T \Vert B^{1/2} \phi(t)\Vert_X^2\, dt \leq k_T \int_0^T \Vert B^{1/2} z(t)\Vert_X^2 dt,
$$
with $k_T = 8T^2\Vert B^{1/2}\Vert^4 +2$.
\item \textbf{Step 3}. Following the optimal weight method introduced by F. Alabau (see, e.g., \cite{alabau_AMO2005}), we set $w(s) = L^{-1}\left( \frac{s}{\beta} \right)$ with $\beta$ appropriately chosen. Nonlinear energy estimate:\\
Prove that
\begin{multline*}
\int_0^T w(E_\phi(0)) \left( \Vert B^{1/2}u(t)\Vert_X^2 + \Vert B^{1/2}F(u(t))\Vert_X^2 \right) dt \\
\lesssim  T \Vert B\Vert  H^*(w(E_\phi(0)))  + \left( w(E_{\phi}(0)) + 1 \right) \int_0^T \langle Bu(t),F(u(t)) \rangle_X \, dt   .
\end{multline*}
\item \textbf{Step 4}. End of the proof:\\
Using the results of the three steps above, we have
\begin{equation*} 
\begin{split}
&
T \Vert B\Vert  H^*(w(E_\phi(0)))  + \left( w(E_{\phi}(0)) + 1 \right) \int_0^T \langle Bu(t),F(u(t)) \rangle_X \, dt  \\
& \gtrsim 
\int_0^T w(E_\phi(0)) \left( \Vert B^{1/2}u(t)\Vert_X^2 + \Vert B^{1/2}F(u(t))\Vert_X^2 \right) dt \qquad\textrm{(Step 3)} \\
& \gtrsim 
\int_0^T w(E_\phi(0)) \Vert B^{1/2}z(t)\Vert_X^2 \, dt  \qquad\qquad\qquad\qquad\qquad\qquad\textrm{(Step 1)} \\
& \gtrsim 
\int_0^T w(E_\phi(0)) \Vert B^{1/2}\phi(t)\Vert_X^2 \, dt   \qquad\qquad\qquad\qquad\qquad\qquad\textrm{(Step 2)} \\
& \gtrsim \mathrm{Cst}\, w(E_{\phi}(0))  E_{\phi}(0) \qquad \qquad\qquad\qquad\qquad\qquad\qquad\qquad\textrm{(uniform observability inequality)}
\end{split}
\end{equation*}
from which we infer that
$$
E_u(T) \leq E_u(0) \left( 1 - \rho_T L^{-1}\left( \frac{E_u(0)}{\beta} \right) \right) ,
$$
and then the exponential decrease is finally established.
\end{itemize}
\end{proof}

\subsection{Space semi-discretizations}
In this section, we define a general space semi-discrete version of \eqref{main_eq}, with the objective of obtaining a theorem similar to Theorem \ref{thm_continuous}, but in this semi-discrete setting, with estimates that are uniform with respect to the mesh parameter.
As we are going to see, uniformity is not true in general, and in order to recover it we add an appropriate extra numerical viscosity term in the numerical scheme.

\subsubsection{Space semi-discretization setting}
We denote by $\triangle x>0$ the space discretization parameter (typically, step size of the mesh), with $0<\triangle x<{\triangle x}_0$, for some fixed ${\triangle x}_0>0$.
We follow \cite{LabbeTrelat,LasieckaTriggiani1} for the setting. Let $(X_{\triangle x})_{0<\triangle x<{\triangle x}_0}$ be a family of finite-dimensional vector spaces ($X_{\triangle x}\sim\R^{N(\triangle x)}$ with $N(\triangle x)\in\N$). 

We use the notations $\lesssim$ and $\simeq$ as before, also meaning that the involved constants are uniform with respect to $\triangle x$.

Let $\beta\in\rho(A)$ (resolvent of $A$). Following \cite{EngelNagel}, we define $X_{1/2} = (\beta\mathrm{id}_X-A)^{-1/2}(X)$, endowed with the norm $\Vert u\Vert_{X_{1/2}} = \Vert (\beta\mathrm{id}_X-A)^{1/2}u\Vert_X$ (for instance, if $A^{1/2}$ is well defined, then $X_{1/2}=D(A^{1/2})$), and we define $X_{-1/2}=X_{1/2}'$ (dual with respect to $X$).

The general semi-discretization setting is the following.
We assume that, for every $\triangle x\in(0,{\triangle x}_0)$, there exist linear mappings $P_{\triangle x}:X_{-1/2}\rightarrow X_{\triangle x}$ and $\widetilde{P}_{\triangle x}:X_{\triangle x}\rightarrow X_{1/2}$ such that $P_{\triangle x}\widetilde{P}_{\triangle x}=\mathrm{id}_{X_{\triangle x}}$, and such that $P_{\triangle x}=\widetilde{P}_{\triangle x}^*$.
We assume that the scheme is convergent, that is,
$\Vert (I-\widetilde{P}_{\triangle x}P_{\triangle x})u \Vert_X \rightarrow 0$ as $\triangle x\rightarrow 0$, for every $u\in X$.
Here, we have implicitly used the canonical injections $D(A) \hookrightarrow X_{1/2} \hookrightarrow X \hookrightarrow X_{-1/2}$ (see \cite{EngelNagel}).

For every $\triangle x\in(0,{\triangle x}_0)$:
\begin{itemize}
\item $X_{\triangle x}$ is endowed with the Euclidean norm $\Vert u_{\triangle x}\Vert_{\triangle x} = \Vert \widetilde{P}_{\triangle x} u_{\triangle x}\Vert_X$, for $u_{\triangle x}\in X_{\triangle x}$. The corresponding scalar product is denoted by $\langle\cdot,\cdot\rangle_{\triangle x}$.\footnote{Note that $\Vert \widetilde{P}_{\triangle x}\Vert_{L(X_{\triangle x},X)}=1$ and that, by the Uniform Boundedness Principle, $\Vert P_{\triangle x}\Vert_{L(X,X_{\triangle x})} \lesssim 1$.}
\item We set\footnote{We implicitly use the canonical extension of the operator $A:X_{1/2}\rightarrow X_{-1/2}$.}
$$
A_{\triangle x}=P_{\triangle x}A\widetilde{P}_{\triangle x},\qquad
B_{\triangle x}=P_{\triangle x}B\widetilde{P}_{\triangle x} .
$$
Since $P_{\triangle x}=\widetilde{P}_{\triangle x}^*$, $A_{\triangle x}$ is skew-symmetric and $B_{\triangle x}$ is symmetric nonnegative
\item We define $F_{\triangle x}:X_{\triangle x}\rightarrow X_{\triangle x}$ by
$$
\forall u_{\triangle x}\in X_{\triangle x}\qquad F_{\triangle x}(u_{\triangle x}) = P_{\triangle x} F(\widetilde{P}_{\triangle x} u_{\triangle x})
$$
\end{itemize}
Note that $B_{\triangle x}$ is uniformly bounded with respect to $\triangle x$, and $F_{\triangle x}$ is  Lipschitz continuous on bounded subsets of $X_{\triangle x}$, uniformly with respect to $\triangle x$.

\medskip

Now, a priori we consider the space semi-discrete approximation of \eqref{main_eq} given by
\begin{equation}\label{waveqSpaceDiscrete_withoutvisc}
u_{\triangle x}'(t)+A_{\triangle x}u_{\triangle x}(t)+B_{\triangle x}F_{\triangle x}(u_{\triangle x}(t))=0,
\end{equation}
and we wonder whether we are able or not to ensure a uniform decay rate of solutions of \eqref{waveqSpaceDiscrete_withoutvisc}.

It happens that the answer is \textbf{NO} in general.
Let us give hereafter an example in the linear case $(F=\mathrm{id}_X)$.

\subsubsection{Finite-difference space semi-discretization of the 1D damped wave equation}
Let us consider the 1D damped Dirichlet wave equation
\begin{equation}\label{1Ddampedwave}
\begin{split}
& y_{tt}-y_{xx}+ay_t=0,\quad 0<x<1,\ t>0,\\
& y(t,0)=y(t,1)=0,
\end{split}
\end{equation}
with $a\geq 0$ measurable and bounded, satisfying $a(x)\geq\alpha>0$ on $\omega\subset(0,1)$ of measure $0<\vert\omega\vert<1$.
This equation fits in our framework with
$$
u = \begin{pmatrix} y \\ y_t \end{pmatrix},\qquad
A = \begin{pmatrix}
0 & \mathrm{id} \\
-\triangle & 0
\end{pmatrix},\qquad
B = a\,\mathrm{id},\qquad
F = \mathrm{id}.
$$
The energy along any solution is
$$
E(t) = \frac{1}{2}\int_0^1 (y_t(t,x)^2+y_x(t,x)^2) \, dx,
$$
and we have
$$
E'(t)=-\int_0^1 a(x) y_t(t,x)^2 \, dx\leq 0.
$$
It is known that $E(t)$ has an exponential decrease. Indeed, by Theorem \ref{thm_Haraux}, the exponential decrease is equivalent to the observability inequality
\begin{equation}\label{obs1139}
\int_0^T\int_\omega \phi_t(t,x)^2 \, dx\, dt \geq C\int_0^1 (\phi_t(0,x)^2+\phi(0,x)^2) \, dx
\end{equation}
for every solution $\phi$ of the conservative equation 
\begin{equation}\label{wave1141}
\phi_{tt}-\phi_{xx}=0,\qquad \phi(t,0)=\phi(t,1)=0,
\end{equation}
and we have the following result.

\begin{proposition}
The observability inequality \eqref{obs1139} holds, for every $T\geq 2$.
\end{proposition}

\begin{proof}
The proof is elementary and can be done by a simple spectral expansion. Any solution $\phi$ of \eqref{wave1141} can be expanded as
$$
\phi(t,x) = \sum_{k=1}^\infty\left( a_k\cos(k\pi t) + \frac{b_k}{k\pi}\sin(k\pi t) \right)\sin(k\pi x),
$$
with
$$
\phi(0,x) = \sum_{k=1}^\infty a_k\sin(k\pi x),\ \
\textrm{and}\ \
\phi_t(0,x) = \sum_{k=1}^\infty b_k \sin(k\pi x).
$$
Then,
$$
\int_0^1 (\phi_t(0,x)^2+\phi(0,x)^2) \, dx = \frac{1}{2}\sum_{k=1}^\infty (a_k^2+b_k^2).
$$
Now, for every $T\geq 2$, we have
$$
\int_0^T\int_\omega \phi_t(t,x)^2 \, dx\, dt \geq \int_0^2\int_\omega \phi_t(t,x)^2 \, dx\, dt, 
$$
and
\begin{equation*}
\begin{split}
\int_0^2\int_\omega \phi_t(t,x)^2 \, dx\, dt &= \sum_{j,k=1}^\infty \int_0^2 (a_j\cos(j\pi t)+b_j\sin(j\pi t))(a_k\cos(k\pi t)+b_k\sin(k\pi t)) \, dt \\
& \qquad\qquad\qquad\qquad \times \int_\omega \sin(j\pi x)\sin(k\pi x) \, dx \\
&= \sum_{j=1}^\infty (a_j^2+b_j^2) \int_\omega \sin^2(j\pi x) \, dx
\end{split}
\end{equation*}
At this step, we could simply conclude by using the fact that the sequence of functions $\sin^2(j\pi x)$ converges weakly to $1/2$, and then infer an estimate from below. But it is interesting to note that we can be more precise by showing the following simple result, valid for any measurable subset $\omega$ of $[0,1]$.

\begin{lemma}
For every measurable subset $\omega\subset[0,1]$, there holds
$$
\forall j\in\N^*\qquad 
\int_\omega \sin^2(j\pi x) \, dx \geq \frac{\vert\omega\vert-\sin\vert\omega\vert}{2},
$$
where $\vert\omega\vert$ is the Lebesgue measure of $\omega$.
\end{lemma}

With this lemma, the observability property follows and we are done.
It remains to prove the lemma.

For a \emph{fixed} integer $j$, let us consider the problem of minimizing the functional
$$ 
K_j(\omega')=\int_{\omega'} \sin^2(j\pi x)\,dx
$$
over the set of all measurable subsets $\omega'\subset[0,1]$ of Lebesgue measure $\vert\omega'\vert=\vert\omega\vert$. Clearly there exists a unique (up to zero measure subsets) optimal set, characterized as a level set of the function $x\mapsto\sin^2(j\pi x)$, which is the set
$$
\omega_j^{\textnormal{inf}}=\left(0,\frac{\vert\omega\vert}{2j}\right)\ \bigcup\ \bigcup_{k=1}^{j-1}\ \left(\frac{k}{j}-\frac{\vert\omega\vert}{2j},\frac{k}{j}+\frac{\vert\omega\vert}{2j}\right)\ \bigcup\ \left(1-\frac{\vert\omega\vert}{2j},\pi\right),
$$
and there holds
$$
\int_{\omega_j^{\textnormal{inf}}}\sin^2(j\pi x)\,dx  =  2j\int_0^{\vert\omega\vert/2j}\sin^2{jx}\,dx
  =  2\int_0^{\vert\omega\vert/2}\sin^2u\, du
  =  \frac{1}{2}(\vert\omega\vert -\sin (\vert\omega\vert)).
$$
Since the quantity is independent of $j$, the lemma follows, and thus the proposition as well.
\end{proof}

Hence, at this step, we know that the energy is exponentially decreasing for the continuous damped wave equation.

Let us now consider the usual finite-difference space semi-discretization model:
\begin{equation*}
\begin{split}
& \ddot y_j - \frac{y_{j+1}-2y_j+y_{j-1}}{(\triangle x)^2} + a_j \dot y_j = 0,\qquad j=1,\ldots,N,\\
& y_0(t)=y_{N+1}(t)=0.
\end{split}
\end{equation*}
The energy of any solution is given by
$$
E_{\triangle x}(t) = \frac{\triangle x}{2} \sum_{j=0}^N \left( (y_j'(t))^2+ \left( \frac{y_{j+1}(t)-y_j(t)}{\triangle x}\right)^2 \right) ,
$$
and we have
$$
E_{\triangle x}'(t) = -\triangle x\,\sum_{j=1}^N a_j (y_j'(t))^2 \leq 0.
$$
We claim that $E_{\triangle x}(t)$ does not decrease uniformly exponentially.

Indeed, otherwise there would hold
$$
E_{\triangle x}(0) \leq C \triangle x\,\sum_{j=1}^N  \int_0^T  a_j (\dot \phi_j(t))^2 \, dt
$$
for every solution of the conservative system
\begin{equation}\label{conserv1155}
\begin{split}
& \ddot \phi_j - \frac{\phi_{j+1}-2\phi_j+\phi_{j-1}}{(\triangle x)^2} = 0,\qquad j=1,\ldots,N,\\
& \phi_0(t)=\phi_{N+1}(t)=0,
\end{split}
\end{equation}
for some time $T>0$ and some \emph{uniform} $C>0$.
But this is known to be wrong.
To be more precise, it is true that, for every fixed value of $\triangle x$, there exists a constant $C_{\triangle x}$, depending on $\triangle x$, such that any solution of \eqref{conserv1155} satisfies
$$
E_{\triangle x}(0) \leq C_{\triangle x}  \triangle x\, \sum_{j=1}^N  \int_0^T a_j (\dot \phi_j(t))^2 \,dt .
$$
This is indeed the usual observability inequality in finite dimension, but with a constant $C_{\triangle x}$ depending on $\triangle x$ (by the way, this observability inequality holds true for any $T>0$, but we are of course interested in taking $T\geq 2$ in a hope to be able to pass to the limit).
Now, the precise negative result is the following.

\begin{proposition}\cite{MaricaZuazua}
We have
$$C_{\triangle x}\underset{\triangle x\rightarrow 0}{\longrightarrow}+\infty.$$
\end{proposition}

In other words, the observability constant is not uniform and observability is lost when we pass to the limit.
The proof of this proposition combines the following facts:
\begin{itemize}
\item Using \emph{gaussian beams} it can be shown that along every bicharacteristic ray there exists a solution of the wave equation whose energy is \emph{localized along this ray}.
\item The \emph{velocity of highfrequency wave packets} for the discrete model \emph{tends to $0$} as $\triangle x$ tends to $0$.
\end{itemize}
Then, for every $T>0$, for $\triangle x>0$ small enough there exist initial data whose corresponding solution is concentrated along a ray that does not reach the observed region $\omega$ within time $T$. The result follows.

\begin{remark}
This proof of nonuniformity is not easy in the case of an internal observation, and relies on microlocal considerations.
It can be noticed that nonuniformity is much easier to prove, in an elementary way, in the case of a boundary observation. We reproduce here the main steps of \cite{InfanteZuazua1999} with  easy-to-do computations:
\begin{enumerate}
\item The eigenelements of the (finite-difference) matrix $\triangle_{\triangle x}$ are
$$ 
e_{\triangle x,j}= \begin{pmatrix}
\vdots \\ 
\sin(j\pi k\triangle x)\\
\vdots 
\end{pmatrix},
\qquad
\lambda_{\triangle x,j} = -\frac{4}{(\triangle x)^2}\sin^2\frac{j\pi \triangle x}{2},\qquad
j=1,\ldots,N .
$$
\item If $\triangle x\rightarrow 0$ then 
$$
\sqrt{\vert \lambda_{\triangle x,n}\vert} -  \sqrt{\vert \lambda_{\triangle x,n-1}\vert} \sim \pi^2 \triangle x .
$$
\item The vector-valued function of time
$$
\phi_{\triangle x}(t) =   \frac{1}{\sqrt{\vert\lambda_{\triangle x,N}\vert}}\left(
\mathrm{e}^{it\sqrt{\vert\lambda_{\triangle x,N}\vert}}e_{\triangle x,N}
- \mathrm{e}^{it\sqrt{\vert\lambda_{\triangle x,N-1}\vert}}e_{\triangle x,N-1}  \right)
$$
is solution of $\ddot \phi_{\triangle x}(t)=\triangle_{\triangle x}\phi_{\triangle x}(t)$, and if $\triangle x$ tends to $0$ then
\begin{itemize}
\item the term $\displaystyle \triangle x\,\sum_{j=0}^N \left( \left\vert \frac{\phi_{j+1}(0)-\phi_j(0)}{\triangle x}\right\vert^2 + \vert\dot\phi_j(0)\vert^2 \right)$ has the order of $1$,
\item the term $\displaystyle  \triangle x \int_0^T \left(\frac{\dot \phi_N(t)}{\triangle x}\right)^2  dt$ has the order of $\triangle x$.
\end{itemize}
\item Therefore $C_{\triangle x}\rightarrow +\infty$ as $\triangle x\rightarrow 0$.
\end{enumerate}
\end{remark}

\paragraph{Remedy for this example.}
A remedy to this lack of uniformity has been proposed in \cite{tebou_zuazua}, which consists of adding in the numerical scheme a suitable extra numerical viscosity term, as follows:
\begin{equation*}
\begin{split}
& \ddot y_j - \frac{y_{j+1}-2y_j+y_{j-1}}{(\triangle x)^2}  + a_j \dot y_j   - (\triangle x)^2 \left( \frac{\dot y_{j+1}-2\dot y_j+\dot y_{j-1}}{(\triangle x)^2} \right)= 0,\qquad j=1,\ldots,N\\
& y_0(t)=y_{N+1}(t)=0 .
\end{split}
\end{equation*}
With matrix notations, this can be written as
\begin{equation*}\label{zuazuavisc}
\ddot y_{\triangle x} - \triangle_{\triangle x} y_{\triangle x}  + a_{\triangle x} \dot y_{\triangle x} - (\triangle x)^2 \triangle_{\triangle x}\dot y_{\triangle x} = 0 .
\end{equation*}
Now the energy decreases according to
$$
\dot E_{\triangle x} = - \triangle x\,\sum_{j=1}^N a_j \dot y_j^2  -(\triangle x)^3 \sum_{j=0}^N \left( \frac{\dot y_{j+1}-\dot y_j}{\triangle x} \right)^2 ,
$$
and we have the following result.

\begin{proposition}[\cite{tebou_zuazua}]
There exist $C>0$ and $\delta>0$ such that, for every $\triangle x$,
$$
E_{\triangle x}(t) \leq C\exp (-\delta t) E_{\triangle x}(0) .
$$
\end{proposition}

In other words, adding an appropriate viscosity term allows to recover a uniform exponential decrease.

An intuitive explanation is that the high frequency eigenvalues are of the order of $\frac{1}{(\triangle x)^2}$. Indeed,
$$ 
\lambda_{\triangle x,j}^2 = -\frac{4}{(\triangle x)^2}\sin^2\frac{j\pi \triangle x}{2},\quad
j=1,\ldots,N.
$$
Then for highfrequencies the viscosity term $(\triangle x)^2\triangle_{\triangle x}\dot y_{\triangle x}$ has the order of $\dot y_{\triangle x}$, and hence it behaves as a \emph{damping on highfrequencies}.
The main result of \cite{tebou_zuazua} shows that it suffices to recover an uniform exponential decay.

Their proof mainly consists of proving the \emph{uniform observability inequality}
$$
E_{\triangle x}(0) \leq C \triangle x\, \sum_{j=1}^N  \int_0^T \left( a_j (\dot \phi_j(t))^2 +(\triangle x)^2\left( \frac{\dot\phi_{j+1}-\dot\phi_j}{\triangle x} \right)^2  \right) dt ,
$$
for every solution of
\begin{equation*}
\begin{split}
& \frac{\ddot \phi_{j+1}+2\ddot \phi_j+\ddot \phi_{j-1}}{4} - \frac{\phi_{j+1}-2\phi_j+\phi_{j-1}}{(\triangle x)^2} = 0 , \qquad j=1,\ldots,N ,\\
& \phi_0(t)=\phi_{N+1}(t)=0 ,
\end{split}
\end{equation*}
and this follows from a meticulous application of a discrete version of the multiplier method.

Note that they also prove the result in a 2D square, and that this was generalized to any regular 2D domain in \cite{MunchPazoto}.

\subsubsection{General results for linear damped equations, with a viscosity operator}
The idea of adding a viscosity in the numerical scheme, used in \cite{tebou_zuazua} in order to recover a uniform exponential decay rate for the finite-difference space semi-discretization of the linear damped wave equation, can be generalized as follows.

We consider \eqref{main_eq} in the linear case $F=\mathrm{id}_X$, that is,
$$
u'(t) + Au(t) + Bu(t) = 0,
$$
with $A:D(A)\rightarrow X$ skew-adjoint and $B=B^*\geq 0$ bounded.
The energy along any solution is
$$
E(t) = \frac{1}{2}\Vert u(t)\Vert_X^2,
$$
and we have
$$
\dot E(t)=-\Vert B^{1/2}u(t)\Vert_X^2.
$$
We consider the space semi-discrete model \emph{with viscosity}:
\begin{equation}\label{linvisc}
\dot u_{\triangle x} + A_{\triangle x}u_{\triangle x} + B_{\triangle x}u_{\triangle x} + (\triangle x)^\sigma\mathcal{V}_{\triangle x} u_{\triangle x} = 0 ,
\end{equation}
where $\mathcal{V}_{\triangle x}:X_{\triangle x}\rightarrow X_{\triangle x}$ is a positive selfadjoint operator, called the \emph{viscosity operator}.

In the previous example (1D damped wave equation), the viscosity term was $- (\triangle x)^2\triangle_{\triangle x} y_{\triangle x}'$, hence
$$
\sigma=2,\qquad
\mathcal{V}_{\triangle x} = \begin{pmatrix}
0 & 0\\
0 & -\triangle_{\triangle x} .
\end{pmatrix} 
$$
Note that it is not positive definite, however it was proved to be sufficient to infer uniformity.

In general, as we are going to see, a typical example of a viscosity operator is $\mathcal{V}_{\triangle x} = \sqrt{A_{\triangle x}^*A_{\triangle x}}$, plus, possibly, $\varepsilon\,\mathrm{id}_{X_{\triangle x}}$ (to make is positive).

In any case, the energy along a solution of \eqref{linvisc}, defined by
$$
E_{u_{\triangle x}}(t)=\frac{1}{2}\Vert u_{\triangle x}(t)\Vert_{\triangle x}^2,
$$
has a derivative in time given by
$$
\dot E_{u_{\triangle x}}(t) = -\Vert B_{\triangle x}^{1/2}u_{\triangle x}(t)\Vert_{\triangle x}^2
-(\triangle x)^\sigma\Vert (\mathcal{V}_{\triangle x})^{1/2} u_{\triangle x}(t) \Vert_{\triangle x}^2 \leq 0.
$$
Now, we have the following result.

\begin{theorem}[\cite{EZ1}]
Assume that the uniform observability inequality
$$
C\Vert\phi_{\triangle x}(0)\Vert^2 \leq \int_0^T \left( \Vert B_{\triangle x}^{1/2}\phi_{\triangle x}(t)\Vert^2 + (\triangle x)^\sigma \Vert \mathcal{V}_{\triangle x}^{1/2} \phi_{\triangle x}(t)\Vert^2 \right) dt
$$
holds for every solution of the conservative system 
$$
\dot\phi_{\triangle x}+A_{\triangle x}\phi_{\triangle x}=0.
$$
Then the uniform exponential decrease holds, i.e., there exist $C>0$ and $\delta>0$ such that 
$$
E_{u_{\triangle x}}(t) \leq C\exp (-\delta t) E_{u_{\triangle x}}(0) ,
$$
for every $\triangle x$, for every solution $u_{\triangle x}$ of \eqref{linvisc}.
\end{theorem}

This theorem says that, in general, adding a viscosity term in the numerical scheme helps to recover a uniform exponential decay, provided a uniform observability inequality holds true for the corresponding conservative equation.
Of course, given a specific equation, the difficulty is to establish the uniform observability inequality. This is a difficult issue, investigated in some particular cases (see \cite{ZuazuaSIREV} for a survey).

Let us mention several other related results.
In \cite{Ervedoza_NM2009}, such uniform properties are established for second-order equations, with a self-adjoint positive operator, semi-discretized in space by means of finite elements on general meshes, not necessarily regular. The results have been generalized in \cite{Miller}.
In \cite{ANVE}, one can find similar results under appropriate spectral gap conditions (in this reference, we can also find results on polynomial decay, see further).
In \cite{RTT_JMPA2006,RTT_COCV2007}, viscosity operators are used for the plate equation, and for general classes of second-order evolution equations under appropriate spectral gap assumptions, combined with the uniform Huang-Pr\"uss conditions derived in \cite{LiuZheng}.

\begin{remark}
In order to recover uniformity, there exist other ways than using viscosities (see \cite{ZuazuaSIREV}:
\begin{itemize}
\item Use an appropriate numerical scheme (such as mixed finite elements, multigrids).

For instance, considering again the 1D damped wave equation \eqref{1Ddampedwave}, the mixed finite element space semi-discretization model is
\begin{equation*}
\begin{split}
& \frac{\ddot y_{j+1}+2\ddot y_j+\ddot y_{j-1}}{4} - \frac{y_{j+1}-2y_j+y_{j-1}}{(\triangle x)^2} + 
\frac{a_{j-1/2}}{4}( \dot y_{j-1}+\dot y_j) +  \frac{a_{j+1/2}}{4}( \dot y_j+\dot y_{j+1})  = 0,
\\
& y_0(t)=y_{N+1}(t)=0,\qquad\qquad\qquad\qquad\qquad\qquad j=1,\ldots,N.
\end{split}
\end{equation*}
The energy of any solution is
$$
E_{\triangle x}(t) = \frac{\triangle x}{2} \sum_{j=0}^N \left( \left( \frac{\dot y_{j+1}(t)+\dot y_j(t)}{2}\right)^2  + \left( \frac{y_{j+1}(t)-y_j(t)}{\triangle x}\right)^2 \right) ,
$$
and one has
$$
\dot E_{\triangle x}(t) = -\triangle x\, \sum_{j=0}^{N-1} a_{j+1/2} \left( \frac{\dot y_{j+1}(t)+\dot y_j(t)}{2}\right)^2 \leq 0 .
$$
It is then proved in \cite{Ervedoza_COCV2010} that $E_{\triangle x}(t)$ decreases exponentially, uniformly with respect to $\triangle x$.

The proof consists of proving that the \emph{uniform} observability inequality:
$$
E_{\triangle x}(0) \leq C \triangle x\, \sum_{j=1}^N  \int_0^T a_{j+1/2} \left( \frac{\dot \phi_{j+1}(t)+\dot \phi_j(t)}{2}\right)^2 dt ,
$$
for any solution of
\begin{equation*}
\begin{split}
& \ddot \phi_j - \frac{\phi_{j+1}-2\phi_j+\phi_{j-1}}{(\triangle x)^2} = 0,\qquad j=1,\ldots,N,\\
& \phi_0(t)=\phi_{N+1}(t)=0 .
\end{split}
\end{equation*}
The proof of this uniform observability inequality is done thanks to a careful spectral analysis, using Ingham's inequality.

The main difference with the finite-difference numerical scheme is that, here, the discrete eigenvalues are given by
$$ 
\lambda_{\triangle x,j}^2 = -\frac{4}{(\triangle x)^2}\tan^2\frac{j\pi h}{2},\quad
j=1,\ldots,N,
$$
and then their behavior at highfrequencies ($j$ close to $N$) is very different from the case of finite differences where the $\tan$ function was replaced with the $\sin$ function: the derivative of the discrete symbol is now bounded below and hence there are not anymore any wavepackets with vanishing speed (see \cite{ZuazuaSIREV}).

\item Filtering highfrequencies.

In the general framework 
$$
\dot u + Au+Bu=0 ,
$$
with $A:D(A)\rightarrow X$ skew-adjoint and $B=B^*\geq 0$ bounded ,
considering the semidiscrete model
$$
\dot u_{\triangle x} + A_{\triangle x}u_{\triangle x}+B_{\triangle x}u_{\triangle x}=0 ,
$$
the idea (well surveyed in \cite{ZuazuaSIREV}) consists of filtering out the highfrequencies, so that the uniform observability inequality
$$
C\Vert\phi_{\triangle x}(0)\Vert^2 \leq \int_0^T \Vert B_{\triangle x}^{1/2}\phi_{\triangle x}(t)\Vert^2 \, dt
$$
holds for every solution of the conservative system $\dot\phi_{\triangle x}=A_{\triangle x}\phi_{\triangle x}$ with
$$
\phi_{\triangle x}(0)\in \mathcal{C}_{\delta/(\triangle x)^\sigma}(A_{\triangle x}) \qquad\qquad\textrm{(spectrally filtered initial data)},
$$
with
$$
\mathcal{C}_{s}(A_{\triangle x}) = \mathrm{Span}\{ e_{\triangle x, j} \ \mid\  \vert\lambda_{\triangle x, j}\vert\leq s  \}.
$$
Here, the vectors $e_{\triangle x, j}$ are eigenvectors of $A_{\triangle x}$ associated with the eigenvalues $\lambda_{\triangle x, j}$.
Then the uniform exponential decrease holds.

\medskip

Actually, as noticed in \cite{EZ1}, we have equivalence of:
\begin{itemize}
\item Every solution of $\dot\phi_{\triangle x}=A_{\triangle x}\phi_{\triangle x}$ satisfies
$$
C_1\Vert\phi_{\triangle x}(0)\Vert^2 \leq \int_0^T \left( \Vert B_{\triangle x}^{1/2}\phi_{\triangle x}(t)\Vert^2 + (\triangle x)^\sigma \Vert \mathcal{V}_{\triangle x}^{1/2} \phi_{\triangle x}(t)\Vert^2 \right) dt .
$$
\item Every solution of $\dot\phi_{\triangle x}=A_{\triangle x}\phi_{\triangle x}$ with $\phi_{\triangle x}(0)\in \mathcal{C}_{\delta/(\triangle x)^\sigma}(A_{\triangle x})$ satisfies
$$
C_2\Vert\phi_{\triangle x}(0)\Vert^2 \leq \int_0^T  \Vert B_{\triangle x}^{1/2}\phi_{\triangle x}(t)\Vert^2 \, dt .
$$
\end{itemize}
\end{itemize}

In conclusion, there are several ways to recover uniform properties for space semi-discretizations:
\begin{enumerate}
\item Add a viscosity term (sometimes called Tychonoff regularization): its role being to damp out the spurious highfrequencies.
\item Use an adapted method (like mixed finite elements, or multi-grid methods), directly leading to an uniform observability inequality.
\item Filter out highfrequencies, by using spectral projections.
\end{enumerate}
In our case, we are going to focus on the use of appropriate viscosity terms, in order to deal more easily with nonlinear terms.
\end{remark}

\subsubsection{Viscosity in general semilinear equations}
We now come back to the study of the general semilinear equation \eqref{main_eq}, and according to the previous analysis, we consider the space semi-discrete approximation of \eqref{main_eq} given by
\begin{equation}\label{waveqSpaceDiscrete}
u_{\triangle x}'(t)+A_{\triangle x}u_{\triangle x}(t)+B_{\triangle x}F_{\triangle x}(u_{\triangle x}(t))+(\triangle x)^\sigma\mathcal{V}_{\triangle x} u_{\triangle x}(t)=0.
\end{equation}
As before, the additional term $(\triangle x)^\sigma\mathcal{V}_{\triangle x}u_{\triangle x}(t)$, with $\sigma>0$, is a \emph{numerical viscosity term} whose role is crucial in order to establish decay estimates that are uniform with respect to $\triangle x$. 
We only assume that $\mathcal{V}_{\triangle x}:X_{\triangle x}\rightarrow X_{\triangle x}$ (viscosity operator) is a positive selfadjoint operator.

Our objective is to be able to guarantee a uniform decay, with the decay rate obtained in Theorem \ref{thm_continuous}.

The energy of a solution $u_{\triangle x}$ of \eqref{waveqSpaceDiscrete} is
\begin{equation*}
E_{u_{\triangle x}}(t)=\frac{1}{2}\Vert u_{\triangle x}(t)\Vert_{\triangle x}^2 ,
\end{equation*}
and we have, as long as the solution is well defined,
\begin{equation*}
E_{u_{\triangle x}}'(t) = -\langle u_{\triangle x}(t),B_{\triangle x}F_{\triangle x}(u_{\triangle x}(t))\rangle_{\triangle x} 
-(\triangle x)^\sigma\Vert (\mathcal{V}_{\triangle x})^{1/2} u_{\triangle x}(t) \Vert_{\triangle x}^2 .
\end{equation*}

We assume that
\begin{equation*}
\langle u_{\triangle x},B_{\triangle x}F_{\triangle x}(u_{\triangle x})\rangle_{\triangle x} 
+ (\triangle x)^\sigma\Vert (\mathcal{V}_{\triangle x})^{1/2} u_{\triangle x} \Vert_{\triangle x}^2
\geq 0,
\end{equation*}
for every $u_{\triangle x}\in X_{\triangle x}$. Hence the energy $E_{u_{\triangle x}}(t)$ is nonincreasing.

For every $f\in L^2(\Omega,\mu)$, we set
\begin{equation*}
\tilde\rho_{\triangle x}(f) = U^{-1}\tilde P_{\triangle x} P_{\triangle x} U \rho(f)
= U^{-1}\tilde P_{\triangle x} P_{\triangle x} F(Uf) .
\end{equation*}
The mapping $\tilde\rho_{\triangle x}$ is the mapping $\rho$ filtered by the ``sampling operator" 
$$
U^{-1}\tilde P_{\triangle x} P_{\triangle x}U = (P_{\triangle x}U)^* P_{\triangle x}U .
$$
By assumption, the latter operator converges pointwise to the identity as $\triangle x\rightarrow  0$, and in many numerical schemes it corresponds to take sampled values of a given function $f$.

We have $\tilde \rho_{\triangle x}(0)=0$.
Setting $f_{\triangle x} = U^{-1}\tilde P_{\triangle x} u_{\triangle x}$, we assume that 
\begin{equation*}
f_{\triangle x} \tilde\rho_{\triangle x}(f_{\triangle x}) \geq 0, 
\end{equation*}
and that
\begin{equation*}
\begin{split}
c_1\, g(\vert f_{\triangle x}(x)\vert) &\leq \vert\tilde\rho_{\triangle x}(f_{\triangle x})(x)\vert \leq c_2\, g^{-1}(\vert f_{\triangle x}(x)\vert) \quad \textrm{for a.e.}\ x\in\Omega\ \textrm{such that}\ \vert f_{\triangle x}(x)\vert\leq 1,    \\
c_1\, \vert f_{\triangle x}(x)\vert & \leq \vert\tilde\rho_{\triangle x}(f_{\triangle x})(x)\vert \leq c_2 \, \vert f_{\triangle x}(x)\vert\qquad\quad\ \textrm{for a.e.}\ x\in\Omega\ \textrm{such that}\ \vert f_{\triangle x}(x)\vert\geq 1 ,
\end{split}
\end{equation*}
for every $u_{\triangle x}\in X_{\triangle x}$, for every $\triangle x\in(0,\triangle x_0)$.

Note that these additional assumptions are valid for many classical numerical schemes, such as finite differences, finite elements, and in more general, for any method based on Lagrange interpolation, in which inequalities or sign conditions are preserved under sampling. But for instance this assumption may fail for spectral methods (global polynomial approximation) in which sign conditions may not be preserved at the nodes of the scheme. 

\begin{theorem}[\cite{AlabauPrivatTrelat}]\label{thm_space}
In addition to the above assumptions, we assume that there exist $T>0$, $\sigma>0$ and $C_T>0$ such that 
\begin{equation}\label{ineqObsDiscrete}
C_T E_{\phi_{\triangle x}}(0)\leq \int_0^T\left(\Vert B_{\triangle x}^{1/2}\phi_{\triangle x}(t)\Vert_{\triangle x}^2+(\triangle x)^\sigma\Vert \mathcal{V}_{\triangle x}^{1/2}\phi_{\triangle x}(t)\Vert_{\triangle x}^{2}\right)\, dt,
\end{equation}
for every solution of $\phi_{\triangle x}'(t)+A_{\triangle x}\phi_{\triangle x}(t)=0$ (uniform observability inequality with viscosity for the space semi-discretized linear conservative equation).

Then, the solutions of \eqref{waveqSpaceDiscrete}, with values in $X_{\triangle x}$, are well defined on $[0,+\infty)$, and the energy of any solution satisfies
$$
E_{u_{\triangle x}}(t) \lesssim T \max(\gamma_1, E_{u_{\triangle x}}(0)) \, L\left( \frac{1}{\psi^{-1}( \gamma_2 t )} \right) ,
$$
for every $t\geq 0$, with $\gamma_1 \simeq \Vert B\Vert /\gamma_2$ and $\gamma_2 \simeq C_T/ T( T^2\Vert B^{1/2}\Vert^4+1)$.
Moreover, under \eqref{condlimsup}, we have the simplified decay rate
$$
E_{u_{\triangle x}}(t) \lesssim T \max(\gamma_1, E_{u_{\triangle x}}(0)) \, (H')^{-1}\left(\frac{\gamma_3}{t}\right),
$$
for every $t\geq 0$, for some positive constant $\gamma_3\simeq 1$.
\end{theorem}

The main assumption above is the uniform observability inequality \eqref{ineqObsDiscrete}, which is not easy to obtain in general, as already discussed. We stress again that, without the viscosity term, this uniform observability inequality does not hold true in general.
As said previously, a typical example of a viscosity operator, for which uniform observability results do exist in the literature, is $\mathcal{V}_{\triangle x} = \sqrt{A_{\triangle x}^*A_{\triangle x}}$.

\subsection{Time semi-discretizations}
In this section, we show how to design a time semi-discrete numerical scheme for \eqref{main_eq}, while keeping uniform decay estimates. 

As a first remark, let us note that a naive explicit discretization
$$
\frac{u^{k+1}-u^{k}}{\Delta t} + A u^k + B F(u^k) = 0,\qquad t_k = k\triangle t,
$$
is never suitable since
\begin{itemize}
\item it is instable and does never satisfy the basic CFL condition (note also that we are going to consider full discretizations in the next subsection),
\item it does not preserve the energy of the conservative part.
\end{itemize}
Hence an implicit discretization is more appropriate.

Following \cite{EZ1}, we choose an implicit mid-point rule:
$$
\frac{u^{k+1}-u^{k}}{\Delta t} + A \left( \frac{u^k+u^{k+1}}{2} \right) + B F \left( \frac{u^k+u^{k+1}}{2} \right) = 0,\qquad\quad t_k=k \triangle t.
$$
Defining the energy by
$$
E_{u^k} = \frac{1}{2}\Vert u^k\Vert _{X}^2  ,
$$
we have
$$
E_{u^{k+1}} - E_{u^k} =  - \triangle t \left\langle B\left(\frac{u^k+ u^{k+1}}{2}\right), F\left(\frac{u^k+ u^{k+1}}{2}\right) \right\rangle_X .
$$
The question is then to determine whether we have or not a uniform decay.

As before, the answer is \textbf{NO} in general, and in order to see that, we are first going to focus on the linear case.

\subsubsection{Linear case}
Assuming that $F=\mathrm{id}_X$, we consider the equation
$$
u' + Au + Bu = 0,
$$
with the time semi-discrete model
$$
\frac{u^{k+1}-u^{k}}{\Delta t} + A \left( \frac{u^k+u^{k+1}}{2} \right) + B \left( \frac{u^k+u^{k+1}}{2} \right) = 0.
$$
We have
$$
E_{u^k} = \frac{1}{2}\Vert u^k\Vert _{X}^2 \qquad\Rightarrow\qquad
\frac{E_{u^{k+1}} - E_{u^k}}{\triangle t} =  - \left\Vert B^{1/2}\left(\frac{u^k+ u^{k+1}}{2}\right) \right\Vert_X^2 \leq 0,
$$
and we investigate the question of knowing whether we have a uniform exponential decrease
$$
E_{u^k}\leq C E_{u^0} \exp(-\delta k \Delta t) .
$$

A negative answer is given in \cite{ZhangZhengZuazua} for a linear damped wave equation. We do not provide the details, that are similar to what has already been said.

Therefore, as before, filtering, or adding an appropriate viscosity term, is required in a hope to recover uniformity.

In \cite{EZ1}, the authors propose to consider the following general numerical scheme:
\begin{equation*}
\left\{\begin{split}
& \frac{\widetilde{u}^{k+1} - u^k}{\triangle t} + A\left(\frac{u^k+ \widetilde{u}^{k+1}}{2}\right) + B \left( \frac{u^k+ \widetilde{u}^{k+1}}{2}\right) = 0 ,  \\
& \frac{  \widetilde{u}^{k+1} - u^{k+1} }{\triangle t} = \mathcal{V}_{\triangle t} u^{k+1}  ,
\end{split}\right.
\end{equation*}
where $\mathcal{V}_{\triangle t}:X\rightarrow X$ ia s positive selfadjoint operator (viscosity operator).
Written in an expansive way, we have
\begin{multline*}
\frac{u^{k+1} - u^k}{\triangle t}  + A\left(\frac{u^k+ u^{k+1}  }{2}\right)  
+ B \left( \frac{u^k+ u^{k+1} }{2}\right) \\
+ \frac{\triangle t}{2} B \mathcal{V}_{\triangle t} u^{k+1} 
+ \mathcal{V}_{\triangle t} u^{k+1}  + \frac{\triangle t}{2} A \mathcal{V}_{\triangle t} u^{k+1} = 0 .
\end{multline*}
Considering the energy $E_{u^k} = \frac{1}{2}\Vert u^k\Vert _{X}^2$, we have
\begin{equation*}
\frac{E_{u^{k+1}} - E_{u^k}}{\triangle t} = -  \left\Vert B^{1/2} \left(\frac{u^k+ \widetilde{u}^{k+1}}{2}\right) \right\Vert_X^2 
-  \Vert (\mathcal{V}_{\triangle t})^{1/2} u^{k+1}\Vert_X^2 - \frac{\triangle t}{2} \Vert \mathcal{V}_{\triangle t} u^{k+1}\Vert_X^2  \leq 0.
\end{equation*}
Now, the main (remarkable) result of \cite{EZ1} is the following.

\begin{theorem}[\cite{EZ1}]
\begin{enumerate}
\item We assume the \emph{uniform observability} inequality with viscosity for the time semi-discretized linear conservative equation with viscosity:
\begin{multline*}
C_T \Vert \phi^0\Vert_X^2 \leq \triangle t \sum_{k=0}^{N-1} \left\Vert  B^{1/2} \left( \frac{\phi^k+ \widetilde{\phi}^{k+1}}{2}\right)\right\Vert_X^2 \\
+ \triangle t \, \sum_{k=0}^{N-1} \Vert (\mathcal{V}_{\triangle t})^{1/2}\phi^{k+1}\Vert_X^2+
(\triangle t )^2 \sum_{k=0}^{N-1} \Vert \mathcal{V}_{\triangle t}\phi^{k+1}\Vert_X^2 
\end{multline*}
for all solutions of
\begin{equation*}
\left\{\begin{split}
& \frac{\widetilde{\phi}^{k+1} - \phi^k}{\triangle t} + A\left(\frac{\phi^k+ \widetilde{\phi}^{k+1}}{2}\right) = 0  \\
& \frac{\widetilde{\phi}^{k+1}-\phi^{k+1}}{\triangle t} = \mathcal{V}_{\triangle t} \phi^{k+1} \end{split}\right.
\end{equation*}
Then the uniform exponential decrease holds, i.e.,
$$
E_{u^k}\leq C E_{u^0} \exp(-\delta k \Delta t),
$$
for some positive constants $C$ and $\delta$ not depending on the solution.
\item The previous uniform observability inequality holds true if
$$
\mathcal{V}_{\triangle t} = -(\triangle t)^2 A^2 = (\triangle t)^2 A^*A .
$$
\end{enumerate}
\end{theorem}

Note that this result provides a choice viscosity operator that always works, in the sense that the uniform observability inequality is systematically valid with this viscosity.
Other choices are possible, such as 
$$
\mathcal{V}_{\triangle t} = - (\mathrm{id}_X- (\triangle t)^2 A^2)^{-1} (\triangle t)^2 A^2 .
$$
The proof is done by resolvent estimates (Hautus test) and by carefully treating lower and higher frequencies.

\subsubsection{Semilinear case}
Coming back to the general semilinear case, we consider the implicit midpoint time discretization of \eqref{main_eq} given by
\begin{equation}\label{time_dis_nl}
\left\{\begin{split}
& \frac{\widetilde{u}^{k+1} - u^k}{\triangle t} + A\left(\frac{u^k+ \widetilde{u}^{k+1}}{2}\right) + BF \left( \frac{u^k+ \widetilde{u}^{k+1}}{2}\right) = 0,\\
& \frac{  \widetilde{u}^{k+1} - u^{k+1} }{\triangle t} = \mathcal{V}_{\triangle t} u^{k+1} , \\
& u^0=u(0) .
\end{split}\right.
\end{equation}
Written in an expansive way, \eqref{time_dis_nl} gives
\begin{equation*}
\frac{u^{k+1} - u^k}{\triangle t}  + A\left(\frac{u^k+ u^{k+1}  }{2}\right)  + BF \left( \frac{u^k+ (\mathrm{id}_X + \triangle t\, \mathcal{V}_{\triangle t}) u^{k+1} }{2}\right) 
+ \mathcal{V}_{\triangle t} u^{k+1}  + \frac{\triangle t}{2} A \mathcal{V}_{\triangle t} u^{k+1} = 0 .
\end{equation*}
For the energy $E_{u^k} = \frac{1}{2}\Vert u^k\Vert _{X}^2$, we have
\begin{multline*}
E_{u^{k+1}} - E_{u^k} = - \triangle t \left\langle B\left(\frac{u^k+ \widetilde{u}^{k+1}}{2}\right), F\left(\frac{u^k+ \widetilde{u}^{k+1}}{2}\right) \right\rangle_X \\
- \triangle t \, \Vert (\mathcal{V}_{\triangle t})^{1/2} u^{k+1}\Vert_X^2 - \frac{(\triangle t )^2}{2} \Vert \mathcal{V}_{\triangle t} u^{k+1}\Vert_X^2 \leq 0, 
\end{multline*}
for every integer $k$.

\begin{theorem}[\cite{AlabauPrivatTrelat}]\label{thm_time}
In addition to the above assumptions, we assume that there exist $T>0$ and $C_T>0$ such that, setting $N=[T/\triangle t]$ (integer part), we have
\begin{multline*}
C_T \Vert \phi^0\Vert_X^2 \leq \triangle t \sum_{k=0}^{N-1} \left\Vert  B^{1/2} \left( \frac{\phi^k+ \widetilde{\phi}^{k+1}}{2}\right)\right\Vert_X^2 \\
+ \triangle t \, \sum_{k=0}^{N-1} \Vert (\mathcal{V}_{\triangle t})^{1/2}\phi^{k+1}\Vert_X^2+
(\triangle t )^2 \sum_{k=0}^{N-1} \Vert \mathcal{V}_{\triangle t}\phi^{k+1}\Vert_X^2 ,
\end{multline*}
for every solution of
\begin{equation*}
\left\{\begin{split}
& \frac{\widetilde{\phi}^{k+1} - \phi^k}{\triangle t} + A\left(\frac{\phi^k+ \widetilde{\phi}^{k+1}}{2}\right) = 0 , \\
& \frac{\widetilde{\phi}^{k+1}-\phi^{k+1}}{\triangle t} = \mathcal{V}_{\triangle t} \phi^{k+1} ,\end{split}\right.
\end{equation*}
(uniform observability inequality with viscosity for the time semi-discretized linear conservative equation with viscosity).

Then, the solutions of \eqref{time_dis_nl} are well defined on $[0,+\infty)$ and, the energy of any solution satisfies
\begin{equation*}
E_{u^{k}} \lesssim  T \max(\gamma_1, E_{u^0}) \, L\left( \frac{1}{\psi^{-1}( \gamma_2 k \triangle t )} \right),
\end{equation*}
for every integer $k$, 
with $\gamma_2\simeq  C_T / T ( 1 + e^{2 T\Vert B\Vert} \max( 1, T\Vert B\Vert ) ) $
and $\gamma_1 \simeq \Vert B\Vert / \gamma_2$.
Moreover, under \eqref{condlimsup}, we have the simplified decay rate
\begin{equation*}
E_{u^{k}} \lesssim  T \max(\gamma_1, E_{u^0}) \, (H')^{-1}\left( \frac{\gamma_3}{k \triangle t } \right),
\end{equation*}
for every integer $k$, for some positive constant $\gamma_3\simeq 1$.
\end{theorem}

\subsection{Full discretizations}
Results for full discretization schemes can be obtained in an automatic way, from the previous time discretization and space discretization results: indeed, following \cite{EZ1}, we note that the results for time semi-discrete approximation schemes are actually valid for a class of abstract systems depending on a parameter, uniformly with respect to this parameter that is typically the space mesh parameter $\triangle x$. Then, using the results obtained for space semi-discretizations, we infer the desired uniform properties for fully discrete schemes. In other words, we first discretize in space and then in time.
We refer to \cite{AlabauPrivatTrelat} for the detailed definition of that abstract class.

\begin{theorem}[\cite{AlabauPrivatTrelat}]\label{thm_full}
Under the assumptions of Theorems \ref{thm_continuous}, \ref{thm_space} and \ref{thm_time}, the solutions of
\begin{equation*}
\left\{\begin{split}
& \frac{\widetilde{u}^{k+1}_{\triangle x} - u^k_{\triangle x}}{\triangle t} + A_{\triangle x}\left(\frac{u^k_{\triangle x}+ \widetilde{u}^{k+1}_{\triangle x}}{2}\right) + B_{\triangle x}F_{\triangle x} \left( \frac{u^k_{\triangle x}+ \widetilde{u}^{k+1}_{\triangle x}}{2}\right) + \mathcal{V}_{\triangle x} \left(\frac{u^k_{\triangle x}+ \widetilde{u}^{k+1}_{\triangle x}}{2}\right) = 0,\\
& \frac{  \widetilde{u}^{k+1}_{\triangle x} - u^{k+1}_{\triangle x} }{\triangle t} = \mathcal{V}_{\triangle t} u^{k+1}_{\triangle x} , 
\end{split}\right.
\end{equation*}
are well defined for every integer $k$, for every initial condition $u^0_{\triangle x}\in X_{\triangle x}$, and for every $\triangle x\in (0,\triangle x_0)$, and the energy of any solution satisfies
\begin{equation*}
\frac{1}{2}\Vert u^k_{\triangle x}\Vert_{\triangle x}^2 = E_{u^{k}_{\triangle x}} \leq  T \max(\gamma_1, E_{u^0_{\triangle x}}) \, L\left( \frac{1}{\psi^{-1}( \gamma_2 k \triangle t )} \right),
\end{equation*}
for every integer $k$, 
with $\gamma_2\simeq  C_T / T ( 1 + e^{2 T\Vert B\Vert} \max( 1, T\Vert B\Vert ) ) $
and $\gamma_1 \simeq \Vert B\Vert / \gamma_2$.
Moreover, under \eqref{condlimsup}, we have the simplified decay rate
\begin{equation*}
E_{u^{k}_{\triangle x}} \leq   T \max(\gamma_1, E_{u^0_{\triangle x}}) \, (H')^{-1}\left( \frac{\gamma_3}{k \triangle t } \right),
\end{equation*}
for every integer $k$, for some positive constant $\gamma_3\simeq 1$.
\end{theorem}

As an example, we consider the nonlinear damped wave equation
$$
\partial_{tt}u(t,x) - \triangle u(t,x) + b(x) \rho(x,\partial_t u(t,x))=0 ,
$$
with appropriate assumptions on $\rho$, as discussed previously.
We first semi-discretize in space with finite differences with the viscosity operator $\mathcal{V}_{\triangle x} = - (\triangle x)^2\triangle_{\triangle x}$, obtaining
$$
u_{\triangle x,\sigma}''(t)-\triangle_{\triangle x} u_{\triangle x,\sigma}(t) + b_{\triangle x}(x_\sigma) \rho(x_\sigma,u_{\triangle x,\sigma}'(t)) - (\triangle x)^2\triangle_{\triangle x} u_{\triangle x}'(t) = 0 ,
$$
where $\sigma$ is an index for the mesh. The solutions of that system have the uniform energy decay rate $L(1/\psi^{-1}(t))$ (up to some constants).

We next discretize in time, obtaining
\begin{equation*}
\left\{\begin{split}
& \frac{\widetilde{u}^{k+1}_{\triangle x,\sigma} - u^k_{\triangle x,\sigma}}{\triangle t} =  \frac{v^k_{\triangle x,\sigma}+\widetilde{v}^{k+1}_{\triangle x,\sigma} }{2}, \\
& \frac{\widetilde{v}^{k+1}_{\triangle x,\sigma} - v^k_{\triangle x,\sigma}}{\triangle t} - \triangle_{\triangle x,\sigma} \frac{u^k_{\triangle x,\sigma}+ \widetilde{u}^{k+1}_{\triangle x,\sigma}}{2} + b_{\triangle x,\sigma}(x_\sigma) \rho \left( x_\sigma, \frac{v^k_{\triangle x,\sigma}+\widetilde{v}^{k+1}_{\triangle x,\sigma} }{2} \right)  \\
& \qquad\qquad\qquad\qquad\qquad\qquad\qquad\qquad\qquad\qquad\qquad
 - (\triangle x)^2 \triangle_{\triangle x,\sigma} \frac{v^k_{\triangle x,\sigma}+\widetilde{v}^{k+1}_{\triangle x,\sigma} }{2} = 0,\\
& \frac{  \widetilde{u}^{k+1}_{\triangle x,\sigma} - u^{k+1}_{\triangle x,\sigma} }{\triangle t} = -(\triangle t)^2 \triangle_{\triangle x,\sigma}^2 u^{k+1}_{\triangle x,\sigma} , \\
& \frac{  \widetilde{v}^{k+1}_{\triangle x,\sigma} - v^{k+1}_{\triangle x,\sigma} }{\triangle t} = -(\triangle t)^2 \triangle_{\triangle x,\sigma}^2 v^{k+1}_{\triangle x,\sigma} ,
\end{split}\right.
\end{equation*}
and according to Theorem \ref{thm_full}, the solutions of that system have the uniform energy decay rate $L(1/\psi^{-1}(t))$ (up to some constants).

\subsection{Conclusion and open problems}
\subsubsection{(Un)-Boundedness of $B$}
We have assumed $B$ bounded:
this involves the case of local or nonlocal internal dampings, but this does not cover, for instance, the case of boundary dampings.

For example, let us consider the linear 1D wave equation with boundary damping
\begin{equation*}
\left\{\begin{array}{ll}
\partial_{tt}u-\partial_{xx} u = 0, & t\in(0,+\infty), \ x\in(0,1), \\
u(t,0)=0,\quad \partial_x u(t,1) = -\alpha \partial_t u(t,1), & t\in(0,+\infty),
\end{array}\right.
\end{equation*}
with $\alpha>0$. The energy
$$
E(t)=\frac{1}{2}\int_0^1 \left( \vert\partial_t u(t,x)\vert^2 + \vert\partial_x u(t,x)\vert^2 \right) dx
$$
decays exponentially.
It is proved in \cite{tebou_zuazua2} that the solutions of
\begin{equation*}
\left\{\begin{split}
& u''_{\triangle x} - \triangle_{\triangle x} u_{\triangle x} - (\triangle x)^2 \triangle_{\triangle x} u'_{\triangle x} = 0, \\
& u_{\triangle x}(t,0)=0,\quad D_{\triangle x} u_{\triangle x}(t,1) = -\alpha \partial_t u(t,1), 
\end{split}\right.
\end{equation*}
(regular finite-difference space semi-discrete model with viscosity) with
$$
D_{\triangle x} = \frac{1}{\triangle x}\begin{pmatrix}
-1 & 1 & \hdots & 0\\
0 & \ddots & \ddots  & \vdots\\
\vdots & \ddots & \ddots & 1\\
0 & \hdots & 0 & -1
\end{pmatrix}
$$
(forward finite-difference operator), have a \emph{uniform exponential decay}. But this fails without the viscosity term.

This example is not covered by our results.

We also mention \cite{BanksItoWang_1991} where a sufficient condition (based on energy considerations and not on viscosities) ensuring uniformity is provided and is applied to 1D mixed finite elements and to polynomial Galerkin approximations in hypercubes. The extension to the semilinear setting is open.

\subsubsection{More general nonlinear models}
If the nonlinearity $F$ involves an unbounded operator, like for instance
$$
u'(t)+Au(t)+BF(u(t),\nabla u(t)) = 0,
$$
then the situation is widely open.

For instance, one can think of investigating semilinear wave equations with strong damping
$$
\partial_{tt}u - \triangle u - a \triangle \partial_t u + b \partial_t u + g\star\triangle u + f(u) =0 ,
$$
with $f$ not too much superlinear. There are many possible variants, with boundary damping, with delay, with nonlocal terms (convolution), etc.
Many results do exist in the continuous setting, but the investigation of whether the uniform decay is preserved under discretization or not remains to be done.

\subsubsection{Energy with nonlinear terms}
Another case which is not covered by our results is the case of semilinear wave equations with locally distributed damping
$$
\partial_{tt}u - \triangle u + a(x) \partial_t u + f(u) =0 ,
$$
with Dirichlet boundary conditions, $a$ a nonnegative bounded function which is positive on an open subset $\omega$ of $\Omega$, and $f$ a function of class $C^1$ such that $f(0)=0$, $sf(s)\geq 0$ for every $s\in\R$ (defocusing case), $\vert f'(s)\vert\leq C\vert s\vert^{p-1}$ with $p\leq n/(n-2)$ (energy subcritical).
It is proved in \cite{Zuazua_CPDE1990} (see also \cite{DehmanLebeauZuazua,JolyLaurent_APDE2013}) that, under geometric conditions on $\omega$, the energy involving a nonlinear term
$$
\int_\Omega \left(  (\partial_t u)^2 + \Vert\nabla u\Vert^2 + F(u) \right) dx ,
$$
with $F(s)=\int_0^s f$,
decays exponentially in time along any solution.
The investigation of the uniform decay for approximation schemes seems to be open.
In particular, it would be interesting to know whether microlocal arguments withstand the discretization procedure.
In particular we raise the following informal question:
\begin{quote}
Do microlocalization and discretization commute?
\end{quote}

\subsubsection{Uniform polynomial energy decay without observability}
Let us assume that $F=\mathrm{id}_X$ in \eqref{main_eq} (linear case).
As proved in Theorem \ref{thm_Haraux}, observability for the conservative linear equation is equivalent to the exponential decay for the linear damped equation.

If observability fails for the conservative equation, then we do not have an exponential decay, however we can still hope to have, for instance, a polynomial decay.
This is the case for some weakly damped wave equations when GCC fails (see \cite{BLR}).

In such cases, it is interesting to investigate the question of proving a uniform polynomial decay rate for space and/or time semi-discrete approximations.
In \cite{ANVE}, such results are stated for second-order linear equations, with appropriate viscosity terms, and under adequate spectral gap conditions.
The extension to a more general framework (weaker assumptions, full discretizations), and to semilinear equations, seems to be open.

\end{document}